\documentclass[reqno]{amsart}
\usepackage{amsthm,amsmath,amssymb}
\usepackage{stmaryrd,mathrsfs}
\usepackage{esint}
\usepackage{tikz}

\usetikzlibrary{patterns,arrows,decorations.pathreplacing}

\allowdisplaybreaks

\newcommand{\ud}[0]{\,\mathrm{d}}

\newcommand{\dist}[0]{\operatorname{dist}}

\newcommand{\abs}[1]{|#1|}
\newcommand{\Babs}[1]{\Big|#1\Big|}

\newcommand{\Norm}[2]{\|#1\|_{#2}}

\newcommand{\ave}[1]{\langle #1\rangle}

\newcommand{\BMO}[0]{\operatorname{BMO}}
\newcommand{\supp}[0]{\operatorname{supp}}

\newcommand{\sign}[0]{\operatorname{sgn}}

\newcommand{\R}{\mathbb{R}}

\newcommand{\N}{\mathbb{N}}

\newcommand{\eps}[0]{\varepsilon}

\swapnumbers \numberwithin{equation}{section}

\theoremstyle{plain}
\newtheorem{theorem}[equation]{Theorem}
\newtheorem{proposition}[equation]{Proposition}
\newtheorem{corollary}[equation]{Corollary}
\newtheorem{lemma}[equation]{Lemma}

\theoremstyle{definition}
\newtheorem{definition}[equation]{Definition}

\theoremstyle{remark}
\newtheorem{remark}[equation]{Remark}

\makeatletter
\@namedef{subjclassname@2010}{%
  \textup{2010} Mathematics Subject Classification}
\makeatother

\begin{document}

\title{The space $JN_p$: nontriviality and duality}

\author[G. Dafni]{Galia Dafni}
\address{(G.D.) Concordia University, Department of Mathematics and Statistics, Montreal, Quebec, H3G-1M8, Canada}
\email{galia.dafni@concordia.ca}

\author[T. Hyt\"onen]{Tuomas Hyt\"onen}
\address{(T.H.) University of Helsinki, Department of Mathematics and Statistics, P.O.B.~68, FI-00014 Helsinki, Finland}
\email{tuomas.hytonen@helsinki.fi}

\author[R. Korte]{Riikka Korte}
\address{(R.K.) Aalto University,
Department of Mathematics and Systems Analysis,
P.O.B. 11100, FI-00076 Aalto, Finland}
\email{riikka.korte@aalto.fi}

\author[H. Yue]{Hong Yue}
\address{(H.Y.)
Georgia College and State University,
Department of Mathematics,
Milledge\-ville, GA 31061, USA}
\email{hong.yue@gcsu.edu}

\thanks{G.D. was partially supported by the Natural Sciences and Engineering Research Council (NSERC) of Canada, the Centre de recherches math\'e{}matiques (CRM) and the Fonds de recherche du Qu\'e{}bec -- Nature et technologies (FRQNT). T.H. was partially supported by the ERC Starting Grant ``AnProb'' (European Research Council grant no. 278558) and the Finnish Centre of Excellence in Analysis and Dynamics Research (Academy of Finland grant nos. 271983 and 307333). R.K. was supported by the Academy of Finland grant no. 308063.  H.Y. was partially supported by GCSU Faculty Development Funds.}

\date{\today}



\maketitle

\begin{abstract}
We study a function space $JN_p$ based on a condition introduced by John and Nirenberg as a variant of $\BMO$. It is known that $L^p\subset JN_{p}\subsetneq L^{p,\infty}$, but otherwise the structure of $JN_p$ is largely a mystery.
Our first main result is the construction of a function that belongs to $JN_p$ but not $L^p$, showing that the two spaces are not the same. Nevertheless, we prove that for monotone functions, the classes $JN_{p}$ and $L^p$ do coincide. Our second main result describes $JN_p$ as the dual of a new Hardy kind of space $HK_{p'}$.
\end{abstract}

\section{Introduction}  

Along with the well-known class of functions of bounded mean oscillation ($\BMO$), John and Nirenberg \cite{JN:61} also introduced 
the following variant of the $\BMO$ condition, which was subsequently used to define what is called the John-Nirenberg space with exponent $p$, denoted by $JN_{p}$. Throughout the paper, it is always understood that $1<p<\infty$. 
 Let $Q_{0}$ be a cube.   As usual we assume cubes have sides parallel to the axes and use $|Q|$ and $\ell(Q)$ to denote the volume and sidelength of $Q$, respectively.

A function $f\in L^1(Q_{0})$ is in $JN_{p}(Q_{0})$ if
\[
\sup\sum_{i}|Q_{i}|\left(\fint_{Q_{i}}|f-\ave f_{Q_{i}}|\right)^p\leq K^p
\]
for some $K<\infty$, where the supremum is taken over all collections of pairwise disjoint cubes $Q_{i}$ in $Q_{0}$, and
 $\ave f_{Q_{i}}$ is the mean of $f$ over $Q_i$.  
We denote the smallest such number $K$ by $\|f\|_{JN_{p}}$.

It is fairly immediate that $L^p\subset JN_p$, but the possibility of equality seems not to have been addressed in the literature. Starting with \cite{JN:61}, several papers \cite{ABKY:11,BKM:16,FPW:98,HMV14,MP:98,MS:16,M:16} prove the inclusion $JN_p\subset L^{p,\infty}$, for the space $JN_p$ as just defined \cite{JN:61}, and for several generalisations or variants of it in the subsequent papers.  Such results would of course trivialise if it turned out that $JN_p$ were just a reformulation of $L^p$. Our first contribution is to show that this is not the case, but that $JN_p$ is indeed a distinct space of its own. While this was probably expected, it does not seem to be completely obvious, even in the one-dimensional setting that we address:

\begin{theorem}\label{thm:main}
Let $p\in(1,\infty)$ and $I\subset\R$ be an interval. Then
\begin{equation}\label{eq:JNpVsLp}
  L^p(I)\subsetneq JN_p(I)\subsetneq L^{p,\infty}(I),
\end{equation}
and $JN_p(I)$ is incomparable with the Lorentz spaces $L^{p,q}(I)$ for $q\in(p,\infty)$. However, the intersections of $L^p(I)$ and $JN_p(I)$ with monotone functions coincide.
\end{theorem}

Theorem \ref{thm:main} indicates that the function space properties of $JN_p$ cannot be immediately deduced from some known results for classical spaces, but require an independent study. Our second main result is the description of $JN_p$ as the Banach space dual of a new ``Hardy kind of'' space $HK_{p'}$. As the details of this duality are somewhat technical, we refer the reader to Section \ref{sec:duality} for a precise statement. Roughly speaking, the space $HK_{p'}$ is defined as an analogue of the atomic description of the Hardy space $H^1$, the well-known predual of $\BMO$; however, reflecting the difference of a supremum over individual cubes in the definition of $\BMO$, and over collections of cubes in $JN_p$, the atoms of $H^1$ are replaced by more complicated structures that we call \emph{polymers} in the definition of $HK_{p'}$. Aside from such technicalities, the proof of our duality result essentially follows a known pattern from the $H^1$-$\BMO$ theory. In contrast to this, the proof of Theorem \ref{thm:main} features phenomena that are new compared to the standard $\BMO$ theory, and we discuss this in some more detail next.

We already pointed out the classical inclusions $L^p\subset JN_p\subset L^{p,\infty}$ \cite{JN:61}. It is also known \cite{ABKY:11} that $JN_p\neq L^{p,\infty}$, which also follows from the fact that $L^p\subsetneq L^{p,\infty}$ and our result  about monotone functions, to be proven in Section \ref{sec:monotone}. The inequality $L^p\neq JN_p$ is established in Section \ref{sec:counter} by exhibiting a concrete example of a function $f\in JN_p\setminus L^p$. The result about monotone functions shows that such a function must necessarily be somewhat complicated. As we check in Remark \ref{rem:Lpq}, the same function also satisfies $f\in JN_p\setminus L^{p,q}$ for every $q<\infty$. On the other hand, the fact that $L^{p,q}\not\subset JN_p$ again follows from the result about monotone functions and the fact that $L^p(I)\subsetneq L^{p,q}(I)$ for $q>p$.

Let us briefly compare Theorem \ref{thm:main} with the well-known limiting case $p=\infty$ corresponding to the space $\BMO$. The analogue of \eqref{eq:JNpVsLp} is $L^\infty\subsetneq \BMO\subsetneq L_{\exp}$, where the second inclusion is the famous John-Nirenberg lemma from \cite{JN:61}, and its strictness is seen e.g.\ by $f(x)=\sign(x)\log\abs{x}$ on $[-1,1]$. For the inequality $L^\infty\subsetneq \BMO$, it suffices to consider the monotone function $f(x)=\log x$ on $[0,1]$, in contrast to the situation of Theorem \ref{thm:main}.

For finite $p<\infty$, the only previously available result related to $L^p\neq JN_p$ is contained in \cite{MP:98}. (This paper, like \cite{FPW:98}, does not explicitly mention the $JN_p$ space, but $JN_p$ is seen as a special case of their more general functionals via the choice $a(B)=\|f\|_{JN_{r}(B)}/|B|^{1/r}$. This connection was observed in \cite{BKM:16}.) However, their counterexample is set on a special metric space (instead of a Euclidean space), which makes it essentially equivalent to a much simpler dyadic situation reproduced in Proposition \ref{prop:dyadic} below.

Despite the number of papers investigating the $JN_p$ space (op.\ cit.), its existing applications seem somewhat limited. After its introduction in \cite{JN:61}, the early papers \cite{C:66,S:65} study these spaces in the context of interpolation of operators. In particular, Campanato \cite{C:66} uses $JN_p$ as a tool to deduce $T:L^p\to L^p$ from the end-points $T:L^\infty\to \BMO$ and $T:L^q\to L^q$, when $q<p<\infty$. This is a widely useful theorem, but its modern proofs do not depend on the $JN_p$ space.

On the other hand, the recently introduced ``$\BMO$-type norms related to the perimeter of sets'' \cite{ABBF:14,ABBF:16,BBM:15}, while not exactly the same as the $JN_p$ norm, have a close similarity, which might motivate a renewed interest in this type of spaces.

As a possible problem for further study, we mention the following: Is there a representation for $JN_p$ functions analogous to the known representation \cite{CR:80} of $\BMO$ functions in the from $\alpha\log(Mg)-\beta\log(Mh)+b$, where $\alpha,\beta\geq 0$ are constants, $g,h$ are positive measurable functions, $M$ is the Hardy-Littlewood maximal operator, and $b$ is a bounded function. Our present contribution does not shed much light on this question.

\begin{remark}
The notation $JN_p$ is borrowed from a number of recent papers, starting with \cite{ABKY:11}, but it is not universal. Stampacchia \cite{S:65} denotes these spaces by $N^{(p,0)}$ (as a special case of certain $N^{(p,\lambda)}$ with a second parameter), whereas Herz \cite{Herz} uses the notation $JN_p$ for a different space (essentially, the $L^p$-based $\BMO$ in a non-classical setting, where the different $L^p$ norms are not necessarily equivalent).
\end{remark}

\subsection*{Acknowledgements.} The authors would like to thank Juha Kinnunen for proposing to study the $JN_{p}$--space. 

\section{Monotone functions}\label{sec:monotone}

Next we show that monotone functions are in $JN_{p}$ if and only if they are in $L^p$. We will first consider bounded intervals. See Remark \ref{remark:unbounded interval} for unbounded intervals.

\begin{theorem}\label{theorem:monotone}
Let $I_{0}\subset \R$ and $f: I_{0}\rightarrow \R$ be a monotone function with $f \in L^1(I_0)$. Then there exists $c=c(p)>0$ such that 
\[
\|f\|_{JN_{p}(I_{0})} \geq c\|f-\ave f_{I_{0}}\|_{L^p(I_{0})}.
\]
\end{theorem}
\begin{proof}
 We can prove the inequality for $I_{0}=[0,1]$ and extend it to any finite interval by using the fact that both sides have the same homogeneity with respect to dilations $f \rightarrow f(\delta \cdot)$. 
Without loss of generality, we may assume that $f:I_{0}\rightarrow\R$ is an increasing function such that $\ave f_{I_{0}}=0$
and that $\|f_{+}\|_{p}\geq \|f_{-}\|_{p}$, where $f_{+}$ and $f_{-}$ are the positive and negative part of $f$, respectively. Moreover, by changing $f$ on a countable
set we can assume it is left-continuous.  

We will first prove the result in the case where  $M=\sup_{I_{0}} f<\infty$. As the constant in the estimate does not depend on $M$, the result easily follows for  unbounded functions as well. This can be seen by considering truncated functions $f_{M}=\max\{\min\{f,M\},-M\}$. As the truncation may not increase the oscillation in any interval, $\|f\|_{JN_{p}(I_{0})} \geq \|f_{M}\|_{JN_{p}(I_{0})}$ for all $0<M<\infty$. The claim now follows by either choosing $M$ large enough so that $\|f_{M}-\ave {f_{M}}_{I_{0}}\|_{L^p(I_{0})}\geq \tfrac12\|f-\ave f_{I_{0}}\|_{L^p(I_{0})}$ in case $f\in L^p(I_{0})$ or by noticing that $\|f_{M}-\ave {f_{M}}_{I_{0}}\|_{L^p(I_{0})}\rightarrow\infty$ as $M\rightarrow\infty$ if $f\notin L^p(I_{0})$.

Let $\lambda_{1}=M$. The other values $\lambda_{k}\in\mathbb R$, $k\in\mathbb N$ or $k\in\{1,2,\ldots, K\}$, will be a decreasing sequence of numbers, defined recursively below, along with sequences of intervals $A_k, B_k, C_k$ and $I_k$. We start with 
$A_{0}=B_{0}=C_{0}=\emptyset$ and for any $k\in\mathbb N$, define the following sets: 
\[
A_{k}=\{x\in I_{0}\,:\, \tfrac{\lambda_{k}}{2}<f(x)\leq \lambda_{k}\}\setminus C_{k-1},
\]
\[
B_{k}=\{x\in I_{0}\,:\, \tfrac{\lambda_{k}}{4}<f(x)\leq \tfrac{\lambda_{k}}{2}\}.
\]

Let us divide the indices into two sets $\mathcal S$ (``\emph{small}'') and $\mathcal G$ (``\emph{good}'') by the following rule: if
\[
|A_{{k}}|\leq 2^{-2p-1}|B_{{k}}|
\]
then we set $k\in \mathcal S$, and otherwise $k\in \mathcal G$.

If ${k}\in \mathcal S$ then we define $I_{k}=A_{k}$, $C_k = \emptyset$, and $\lambda_{k+1}=\lambda_{k}/2$, i.e. $A_{{k+1}}=B_{k}$.

If ${k}\in \mathcal G$ then we define
\[
I_{k}=A_{{k}}\cup B_{{k}}\cup C_{{k}},
\]
where $C_{k} = [c_k,d_k]$ is the interval  adjacent to $B_k$ on the left (i.e.\ $d_k$ is the left endpoint of $B_k$) with $|C_{k}|=|A_{k}|$, and 
\[
\lambda_{k+1}=f(c_{{k}})
\]
 (the continuity of $f$ from the left guaranteeing that  $A_{k+1}$ will not be of zero length),
unless $f(c_{{k}})<0$  or  $c_{{k}} \notin I_{0}$, in which case we stop the process of constructing intervals. Let $K$ be the index where the construction stops.

If $k\in \mathcal S$, then 
\[
\int_{A_{{k}}}f^p\leq \lambda_{k}^p|A_{{k}}|\leq \lambda_{k}^p2^{-2p-1}|B_{{k}}|= \tfrac12 \left(\tfrac{\lambda_{k}}{4}\right)^p|B_{{k}}|\leq \tfrac12\int_{B_{{k}}}f^p=\tfrac12\int_{A_{{k+1}}}f^p.
\]
Thus if $\{m-n,m-n+1,\ldots,m-1\}\subset \mathcal S$ and $m\in\mathcal G$, then
\[\sum_{k=m-n}^{m-1}\int_{A_{k}}f^p\leq \int_{A_{m}}f^p.
\]
If $k\in\mathcal S$ for all $k\geq l$, then let $k_{0}>l$ be any index such that $\lambda_{k_{0}}<1$. Recall that $\lambda_{k_{0}}>0$ always. Then by using the previous estimate, 
as well as the assumptions $\ave f_{I_{0}}=0$ and $|I_0| = 1$, we see that
\[
\sum_{k=l}^{k_{0}-1}\int_{A_{k}}f^p\leq \int_{A_{k_{0}}}f^p\leq |A_{k_{0}}|\lambda_{k_{0}}^p\leq  |A_{k_{0}}|\lambda_{k_{0}}\leq2\int_{I_{0}} f_{+} = \|f\|_{L^1(I_0)}\leq\|f\|_{JN_{p}(I_{0})}.
\]
Letting $k_{0}\rightarrow\infty$, we see that
\[
\sum_{k=l}^{\infty}\int_{A_{k}}f^p\leq \|f\|_{JN_{p}(I_{0})}.
\]
Since $I_k = A_k$ for $k \in  \mathcal S$ and $I_k \supset A_k$ for $k \in  \mathcal G$, the estimates above give
\begin{equation}
\label{small-k}
\sum_{k \in \mathcal S}\int_{I_{k}}f^p \leq \sum_{k \in \mathcal G}\int_{I_{k}}f_{+}^p
+  \|f\|_{JN_{p}(I_{0})}.
\end{equation}

If ${k}\in \mathcal G$ (and $c_{{k}} \in I_{0}$ and $f(c_{k})\geq 0$), then $|A_{k}|=|C_{k}|\geq2^{-2p-1}|B_{k}|$ and therefore $|I_{k}|\leq 2^{2p+2}|A_{k}|$.
Notice that for any $x \in A_k$, $y \in C_k$, $|f(x) - \langle f\rangle_{I_k}| + |\langle f\rangle_{I_k} - f(y)| \geq f(x) - f(y) \geq \frac{\lambda}{2} - \frac{\lambda}{4}$ and therefore
\[
\int_{A_{k}}|f-\ave f_{I_{k}}| + \int_{C_{k}}|f-\ave f_{I_{k}}| \geq (\inf_{A_{{k}}} |f-\ave f_{I_{k}}|+\inf_{C_{{k}}}|\ave f_{I_{k}}  - f|) |A_k| \geq \tfrac{\lambda_{k}}{4}|A_k|.
\]
Now we can estimate
\[
\begin{split}
|I_{k}|\left(\fint_{I_{k}}|f-\ave f_{I_{k}}|\right)^p&\geq |I_{k}|^{1-p}\left( \tfrac{\lambda_{k}}{4} |A_{\lambda_{k}}|  \right)^p\\
& \geq \left( 2^{-2p-4}\right)^p |I_{k}|\lambda_{k}^p \\
&\geq  2^{-2p^2-4p}\int_{I_{k}}f_{+}^p.
\end{split}
\]

Let us consider the last constructed interval in case the construction stops at some point. (It is also possible that the  construction gives us infinitely many intervals.) Recall that we stop the construction when either $f(c_{{K}})<0$ or $c_{{K}}\notin I_{0}$. If $c_{{K}}\in I_{0}$, then we can estimate the integrals on this interval in the same way as for other indices in $\mathcal G$.

If $c_{{K}}\notin I_{0}$, then we have to consider a shorter interval $C_{{K}}=[a,d_{K}]$, where $d_{K}$ is chosen as usual. In this case
\[
\int_{I_{K}}f\leq 0
\]
as $\{f<0\}\cap I_{0}\subset I_K$ and $\ave f_{I_{0}}=0$. Since $f\geq \tfrac{\lambda_{K}}{4}$ in $A_{{K}}\cup B_{{K}}$, we have
\[
\begin{split}
|I_{k}|\left(\fint_{I_{K}}|f-\ave f_{I_{K}}|\right)^p&\geq |I_{K}|^{1-p}\left( \tfrac{\lambda_{K}}{4}(|A_{{K}}|+|B_{{K}}|)  \right)^p\\
& \geq  |I_{K}|^{1-p}\left( \tfrac{\lambda_{K}}{8} |I_{K}| \right)^p\\
&\geq |I_{K}|2^{-3p}\lambda_{K}^p\geq 2^{-3p}\int_{I_{K}}f_{+}^p.
\end{split}
\]
If we include $K$ in $\mathcal G$, we can combine the two estimates above to get
\begin{equation}
\label{good-k}
\sum_{k\in \mathcal G}c_{p}\int_{I_{k}}f_{+}^p \leq \sum_{k\in\mathcal G} |I_{k}|\left(\fint_{I_{k}}|f-\ave f_{I_{k}}|\right)^p, 
\end{equation}
with $c_{p}=2^{-2p^2-3p}$.

Using  (\ref{small-k}) and (\ref{good-k}), and noting that the intervals $\{I_k\}$ are pairwise disjoint and their
union covers $\{f > 0\}$, we conclude that
 \[
\frac 1 2\int_{I_{0}}|f|^p \leq \int_{I_{0}}f_{+}^p = \sum_{k\in \mathcal G}\int_{I_{k}}f_{+}^p +   \sum_{k\in \mathcal S}\int_{I_{k}}f_{+}^p
\leq (2 c_p^{-1}+  1)\|f\|_{JN_{p}(I_{0})}.
\]
\end{proof}

\begin{remark}\label{remark:unbounded interval}
The result also holds for unbounded intervals. Indeed, let $I$ be such an interval (i.e., either the full line $\R$ or a half-line), and $I_n$ be an increasing sequence of finite intervals converging to $I$, say $I_n:=I\cap[-n,n]$. By the result for bounded intervals and elementary monotonicity properties of the norms, we have
\begin{equation*}
  \Norm{f-\ave{f}_{I_n}}{L^p(I_m)}\leq \Norm{f-\ave{f}_{I_n}}{L^p(I_n)}\lesssim\Norm{f}{JN_p(I_n)}\leq\Norm{f}{JN_p(I)}
\end{equation*}
for all $m\leq n$. Here and below the symbol $\lesssim$ indicates the presence of constants in the inequality. Thus
\begin{equation*}
\begin{split}
  \abs{\ave{f}_{I_m}-\ave{f}_{I_n}}
  &=\abs{I_m}^{-1/p}\Norm{\ave{f}_{I_m}-\ave{f}_{I_n}}{L^p(I_m)} \\
  &\leq\abs{I_m}^{-1/p}(\Norm{\ave{f}_{I_m}-f}{L^p(I_m)}+\Norm{f-\ave{f}_{I_n}}{L^p(I_m)}) \\
  &\lesssim\abs{I_m}^{-1/p}\Norm{f}{JN_p(I)}\to 0
\end{split}
\end{equation*}
as $m\to\infty$. Hence $\ave{f}_{I_n}$ converges to some limit $c\in\R$. Hence by Fatou's lemma
\begin{equation*}
  \int_{I}\abs{f-c}^p
  =\int_{I}\lim_{n\to\infty}1_{I_n}\abs{f-\ave{f}_{I_n}}^p
  \leq\liminf_{n\to\infty}\int_{I_n}\abs{f-\ave{f}_{I_n}}^p
  \lesssim\Norm{f}{JN_p(I)}^p,
\end{equation*}
so that indeed $\Norm{f-c}{L^p}\lesssim\Norm{f}{JN_p}$ as claimed.

We note that the case when $I=\R$ can be obtained by a more direct argument, which does not rely on the considerations in the case of finite intervals: Consider a monotone function $f:\R\rightarrow\R$. If $f$ is a constant function, then clearly $f-c\in L^p(\R)\cap JN_{p}(\R)$ and any other monotone functions are not $L^p$-integrable. If $f$ is not constant, we may assume that it is increasing. Then there exists some $\delta>0$ and $a\in\R$ such that $f(a+1)-f(a-1)>2\delta$. Let $I_{k}=[a-k-1,a+k+1]$. Now $|f-\ave f_{I_{k}}|\geq\delta$ either in $[a-k-1,a-1]$ or $[a+1,a+k+1]$. Thus for any $k\in\N$,
\[
\|f\|_{JN_{p}}\geq |I_{k}|\left(\fint_{I_{k}}|f-\ave f_{I_{k}}|dx\right)^p\geq |I_{k}|\left(\frac{\delta\cdot k}{|I_{k}|}\right)^p=\delta^p\frac{k^p}{(2k+2)^{p-1}}.
\]
Letting $k\rightarrow\infty$, we see that $\|f\|_{JN_{p}}=\infty$.
\end{remark}

\section{The counterexample}\label{sec:counter}

One can define a dyadic counterpart of the John-Nirenberg space $JN_{p,dyadic}$ in a natural way by taking the supremum over pairwise disjoint collections of dyadic cubes, see for example \cite{BKM:16} for more details.
It is rather easy to find an example of a function in $JN_{p}\setminus L^p$ in the dyadic case. 
We start with this easy example even though this idea does not work when the $JN_{p}$-norm is taken over all cubes.
For another example with the same idea, see Section 5 in \cite{MP:98}.

\begin{proposition}\label{prop:dyadic}
There exists $f\in JN_{p,dyadic}([0,1])\setminus L^p([0,1]).$
\end{proposition}
\begin{proof}
Let $f:(0,1)\rightarrow\R$ be defined as follows:
\[
f(x)=
2^{k/p},\quad \text{if }2^{-k}\leq x< 2^{-k+1},\quad k=1,2,\ldots\\
\]
Now $f\notin L^p([0,1])$ since
\[
\int_{0}^{1}f^{p}dx=\sum_{k=1}^{\infty}((2^{k/p})^{p}(2^{-k+1}-2^{-k})=\sum_{k=1}^{\infty}1=\infty.
\]

Let $I\subset [0,1]$ be a dyadic interval. We see that $f$ is constant in $I$ unless $I=I_{k}=[0,2^{-k}] $ for some $k\in \N$. Thus there can be at most one non-zero term in the sum of $JN_{p}$-norm and
\[
\begin{split}
&|I_{k}|\left(\fint_{I_{k}}|f-\ave f_{I_{k}}|dx\right)^p\\
\leq &2^{-k}\left(2^{k+1}\int_{0}^{2^{-k}}f\, dx \right)^p\\
=& 2^{-k}\left(2^{k+1}\sum_{m=k+1}^\infty 2^{m/p}2^{-m} \right)^p\\
=& 2^{-k}\left(2^{k+1} \frac{2^{k(1/p-1)}}{2^{(1-1/p)}-1} \right)^p\leq c_{p}\\
\end{split}
\]
Thus $f$ is in dyadic $JN_{p}$.
\end{proof}

Now we give an example in the general case.

\begin{proposition}\label{prop:badf}
There exists a function $f\in JN_p\setminus L^p$.
\end{proposition}

We construct a family of functions $f_I$ indexed by $I\in\mathscr{D}$, the dyadic subintervals of $[0,1)$. We want to point out that $f_{I}$ is not the mean value of $f$ (recall
that for the mean value of $f$ on an interval $J$ we use the notation $\ave f_{J}$); moreover, the functions $f_I$ are not supported in the intervals
$I$ but rather in corresponding intervals $\hat{I}$ which will be defined shortly, and are not in any way assumed
to be dyadic.  The structure of the dyadic intervals is used only in order to simplify the construction (for example
instead of using indices $i,j$, we index by the interval $I = [2^{-i}j, 2^{-i}(j+1))$).  We denote the length of a generic
interval $I\in\mathscr D$ by $\abs{I}=2^{-i}$, without always mentioning this relation of $I$ and $i$ explicitly.

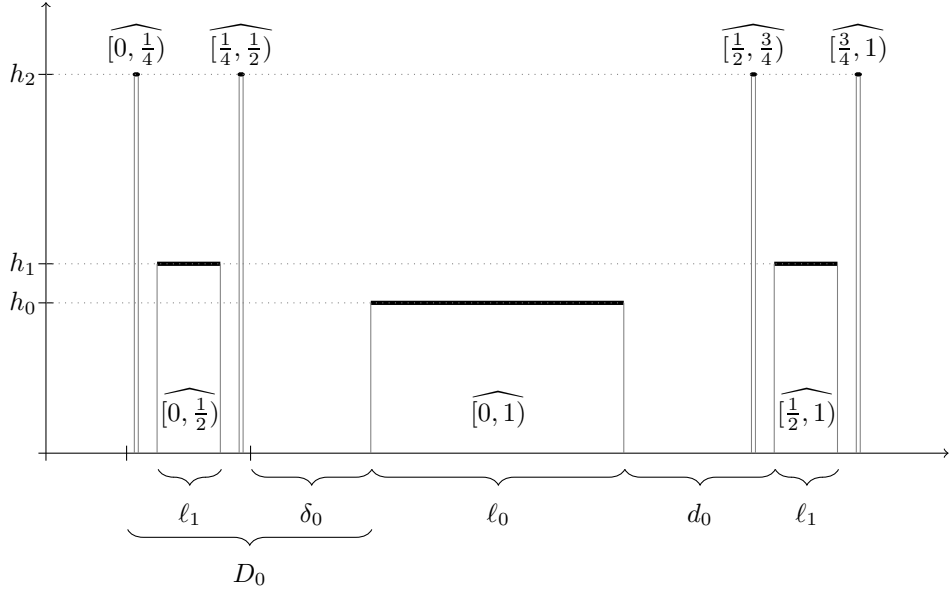
\begin{figure}
\begin{tikzpicture}
\draw[black,->] (-6.1,0) -- (6,0);
\draw[thin,->] (-6,-0.1) -- (-6,6) ;

\draw[black, ultra thick] (4*0.42044820762,2*1)--(4*-0.42044820762,2*1);
\draw[gray] (4*0.42044820762,0)--(4*0.42044820762,2*1);
\draw[gray] (4*-0.42044820762,0)--(4*-0.42044820762,2*1);

\draw (-6.1, 2*1)--(-5.9,2*1);
\draw (-6,2*1) node[left] {$h_{0}$};
\draw[dotted,gray] (-6,2*1)--(4*0.42044820762,2*1);


\draw[black, ultra thick] (4*0.92044820762,2*1.25992104989)--(4*1.13067231144,2*1.25992104989);
\draw[gray] (4*0.92044820762,0)--(4*0.92044820762,2*1.25992104989);
\draw[gray] (4*1.13067231144,0)--(4*1.13067231144,2*1.25992104989);

\draw[black, ultra thick] (4*-0.92044820762,2*1.25992104989)--(4*-1.13067231144,2*1.25992104989);
\draw[gray] (4*-0.92044820762,0)--(4*-0.92044820762,2*1.25992104989);
\draw[gray] (4*-1.13067231144,0)--(4*-1.13067231144,2*1.25992104989);

\draw (-6.1, 2*1.25992104989)--(-5.9,2*1.25992104989);
\draw (-6,2*1.25992104989) node[left] {$h_{1}$};
\draw[dotted,gray] (-6,2*1.25992104989)--(4*1.13067231144,2*1.25992104989);

\draw[black,ultra thick](4*1.19317231144,2*2.51984209979)--(4*1.20631131793,2*2.51984209979) node[black,midway,yshift=0.4cm] {$\widehat{[\tfrac34,1)}$};
\draw[gray] (4*1.19317231144,0)--(4*1.19317231144,2*2.51984209979);
\draw[gray] (4*1.20631131793,0)--(4*1.20631131793,2*2.51984209979);
\draw[black, fill=black] (4*1.19317231144,2*2.51984209979) circle (.015cm);
\draw[black, fill=black] (4*1.20631131793,2*2.51984209979) circle (.015cm);

\draw[black,ultra thick](-4*1.19317231144,2*2.51984209979)--(-4*1.20631131793,2*2.51984209979)  node[black,midway,yshift=0.4cm] {$\widehat{[0,\tfrac14)}$}  ;
\draw[gray] (-4*1.19317231144,0)--(-4*1.19317231144,2*2.51984209979);
\draw[gray] (-4*1.20631131793,0)--(-4*1.20631131793,2*2.51984209979);
\draw[black, fill=black] (-4*1.19317231144,2*2.51984209979) circle (.015cm);
\draw[black, fill=black] (-4*1.20631131793,2*2.51984209979) circle (.015cm);

\draw[black, ultra thick](4*0.85794820762,2*2.51984209979)--(4*0.84480920113,2*2.51984209979) node[black,midway,yshift=0.4cm] {$\widehat{[\tfrac12,\tfrac34)}$};
\draw[gray, ultra thin] (4*0.85794820762,0)--(4*0.85794820762,2*2.51984209979);
\draw[gray, thin] (4*0.84480920113,0)--(4*0.84480920113,2*2.51984209979);
\draw[black, fill=black] (4*0.85794820762,2*2.51984209979) circle (.015cm);
\draw[black, fill=black] (4*0.84480920113,2*2.51984209979) circle (.015cm);

\draw[black, ultra thick](-4*0.85794820762,2*2.51984209979)--(-4*0.84480920113,2*2.51984209979)  node[black,midway,yshift=0.4cm] {$\widehat{[\tfrac14,\tfrac12)}$};
\draw[gray] (-4*0.85794820762,0)--(-4*0.85794820762,2*2.51984209979);
\draw[gray] (-4*0.84480920113,0)--(-4*0.84480920113,2*2.51984209979);

\draw[black, fill=black] (-4*0.85794820762,2*2.51984209979) circle (.015cm);
\draw[black, fill=black] (-4*0.84480920113,2*2.51984209979) circle (.015cm);

\draw (-6.1, 2*2.51984209979)--(-5.9,2*2.51984209979);
\draw (-6,2*2.51984209979) node[left] {$h_{2}$};
\draw[dotted,gray] (-6,2*2.51984209979)--(4*1.19317231144,2*2.51984209979);

%

\draw [yshift=-0.2cm,decorate,decoration={brace,amplitude=6pt,mirror},xshift=0.4pt,yshift=-0.4pt](-4*0.42044820762,0) -- (4*0.42044820762,0) node[black,midway,yshift=-0.6cm] { $\ell_{0}$} node[black,midway,yshift=0.8cm] {$\widehat{[0,1)}$};

\draw [yshift=-0.2cm,decorate,decoration={brace,amplitude=6pt,mirror},xshift=0.4pt,yshift=-0.4pt](4*0.92044820762,0) -- (4*1.13067231144,0) node[black,midway,yshift=-0.6cm] { $\ell_{1}$}  node[black,midway,yshift=0.8cm] {$\widehat{[\tfrac12,1)}$} ;

\draw [yshift=-0.2cm,decorate,decoration={brace,amplitude=6pt,mirror},xshift=0.4pt,yshift=-0.4pt](-4*1.13067231144,0) --(-4*0.92044820762,0) node[black,midway,yshift=-0.6cm] { $\ell_{1}$}  node[black,midway,yshift=0.8cm] {$\widehat{[0,\tfrac12)}$} ;

\draw [yshift=-0.2cm,decorate,decoration={brace,amplitude=6pt,mirror},xshift=0.4pt,yshift=-0.4pt](-4*0.82,0) --(-4*0.42044820762,0) node[black,midway,yshift=-0.6cm] { $\delta_{0}$}   ;
\draw [black]  (-4*0.82,-0.1)--(-4*0.82,0.1);

\draw [yshift=-1cm,decorate,decoration={brace,amplitude=6pt,mirror},xshift=0.4pt,yshift=-0.4pt](-4*1.232,0) --(-4*0.42044820762,0) node[black,midway,yshift=-0.6cm] { $D_{0}$}   ;
\draw [black]  (-4*1.232,-0.1)--(-4*1.232,0.1);

\draw [yshift=-0.2cm,decorate,decoration={brace,amplitude=6pt,mirror},xshift=0.4pt,yshift=-0.4pt](4*0.42044820762,0) -- (4*0.92044820762,0) node[black,midway,yshift=-0.6cm] { $d_{0}$};

\end{tikzpicture}

\caption{\label{picture}First three generations ($i=0,1,2$) in the construction of $f$ (with $p=3$). The distances $\delta_{i}$ and $D_{i}$ are defined in Lemma \ref{lemma:delta}; in particular, $\delta_{0}$ (resp.\ $D_{0}$) is the distance from $\widehat{[0,1)}$ to the nearest (resp.\ farthest) $\hat{I}$ after infinitely many iterations of the construction.} 
\end{figure}

For every $I\in\mathscr D$, we define the numbers
\begin{itemize}
	\item (``length'') $\ell_I:=\ell_i:=2^{-(i+1/2)^2}$,
	\item (``distance'') $d_I:=d_i:=2^{-(i+1)^2}$,
	\item (``height'') $h_I:=h_i:=2^{i^2/p}$.
\end{itemize}

For every $I\in\mathscr D$, we define another interval $\hat I$ recursively as follows:
\begin{itemize}
	\item For $I^0:=[0,1)$, let $\hat I^0$ be some interval of length $\ell_{0}$.
	\item If $\hat I$ is already defined and $I'$ is the left (right) half of $I$, let $\hat I'$ be the interval of length $\ell_{I'}$ positioned on the left (right) side of $\hat I$ in such a way that $\dist(\hat I',\hat I)=d_I$.
\end{itemize}

Finally, let
\begin{itemize}
	\item $f_I:=h_I\cdot 1_{\hat I}$,
	\item $f_I^*:=\sum_{I'\subseteq I}f_{I'}$,
	\item $f:=f_{[0,1)}^*$.
\end{itemize}
We will use the word descendants of $\hat{I}$
to refer to $\hat{I'}$ for any proper subinterval $I' \subset I$, and similarly for the corresponding functions $f_I$.

See Figure \ref{picture} for the first steps in the construction.  While the
function $f$ is defined on the whole line, it is in fact supported in a finite interval of length $\ell_0 + 2D_0$, as will be seen 
in the following lemma.

\begin{lemma}\label{lemma:delta}
Let $I'$ be the left or right half of $I$, and define the distances
\begin{equation*}
\begin{split}
  \delta_i &=\delta_I:=\dist(\hat I,\supp f_{I'}^*)
  =\inf_{x\in \supp f_{I'}^*}\dist(\hat I,x), \\
  D_i &=D_I:=\sup_{x\in \supp f_{I'}^*}\dist(\hat I,x).
\end{split}
\end{equation*}
Then
\begin{equation*}
    \frac12 d_i\leq \delta_i\leq d_i\leq D_i\leq 3d_i.
\end{equation*}
In particular, all the functions $f_I$ are disjointly supported.
\end{lemma}

\begin{proof}
Note that $\supp f_{I'}^*$ is the union of all $\hat I''$, where $I''\subseteq I'$. Thus $D_I$ is the maximal (or supremal) distance of $\hat I$ from any point of $\hat I''$, among all dyadic descendants $I''$ of $I'$. This maximal distance is achieved by considering descendants always on the same side of $\hat I$. With this in mind, let $I_0:=I$, $I_1:=I'$ and recursively $I_k$ be the left (right) half of $I_{k-1}$ if $I'$ is the left (right) half of $I$. Then the distance $D_I$ is formed by subsequently summing up the distance between two consecutive intervals $\hat I_{k}$ and $\hat I_{k+1}$, and the length of the next interval $\hat I_{k+1}$.
Thus
\begin{equation*}
\begin{split}
  D_I &=\sum_{k\geq 0}( \dist(\hat I_{k},\hat I_{k+1})+\abs{\hat I_{k+1}} )
    =\sum_{k\geq 0} (d_{I_{k}}+\ell_{I_{k+1}}) \\
 &=\sum_{k\geq 0} (d_{i+k}+\ell_{i+k+1})
   =\sum_{k\geq 0} (2^{-(i+1+k)^2}+2^{-(i+1+k+1/2)^2}) \\
  &\leq 2^{-(i+1)^2}\sum_{k\geq 0}(2^{-k^2}+2^{-(k+1/2)^2})
    \leq 3\cdot 2^{-(i+1)^2}=3 d_i.
\end{split}
\end{equation*}

Similarly, $\delta_I$ is the minimal (or infimal) distance of $\hat I$ from any point of $\hat I''$, among all dyadic descendants $I''$ of $I'$. If we take, without loss of generality, 
$I'$ to be the left half of $I$, this minimal distance is achieved by always
choosing the descendants $I''$ of $I'$ on the right, so that $\hat{I''}$ gets closer and closer to $\hat{I}$.  From the computation above,
the maximum distance that points in these $\hat{I''}$ lie to the right of $\hat{I'}$ is $D_{I'}$, and therefore, recalling that $d_I$ is the
distance from $\hat{I}$ to $\hat{I'}$, 
\begin{equation*}
\begin{split}
  \delta_I=d_I-D_{I'}
 & \geq d_i-3d_{i+1}
  =2^{-(i+1)^2}-3\cdot 2^{-(i+2)^2} \\
 & \geq 2^{-(i+1)^2}(1-3\cdot 2^{-2(i+1)-1})
 \geq 2^{-(i+1)^2}(1-3/8)>\tfrac12 d_i.
\end{split}
\end{equation*}

We finally come to the claim concerning disjoint supports of the $f_I$. Since $\delta_I>0$ is the minimal distance of $\supp f_I$ from any $\supp f_{I''}$ with $I''\subsetneq I$, we see that each $f_I$ is disjointly supported from its descendants $f_{I''}$ with $I''\subsetneq I$. If, on the other hand, $I_1,I_2\subset[0,1)$ are two disjoint dyadic intervals, then we can find the smallest dyadic interval $I$ that contains both $I_1$ and $I_2$, and thus (possibly after reindexing) $I_1$ must be contained in the left half of $I$ and $I_2$ in the right half. But then $\hat I_1$ lies on the left side of $\hat I$ and $\hat I_2$ on the right side, so that clearly $\hat I_1$ and $\hat I_2$ are disjoint. Since any two dyadic intervals are either disjoint or one is a descendant of the other one, we have checked the disjointness of the supports $\supp f_I=\hat I$ in all cases.
\end{proof}

\begin{lemma}
	For all $I\in\mathscr D$, we have
	\begin{itemize}
		\item $\Norm{f_I}{1}=2^{-i^2/p'-i-1/4}\leq\Norm{f_I^*}{1}\leq c_p\Norm{f_I}{1}$, where $p':=p/(p-1)$, and
		\item $\Norm{f_I^*}{p}=\infty$, so in particular $\Norm{f}{p}=\infty$.
	\end{itemize}
\end{lemma}

\begin{proof}
	Clearly
	\begin{equation*}
	\Norm{f_I^*}{1}\geq \Norm{f_I}{1}=h_i\ell_i=2^{i^2/p}2^{-i^2-i-1/4}
	=2^{-i^2/p'-i-1/4}.
	\end{equation*}
	Hence
	\begin{equation*}
	\Norm{f_I^*}{1}
	=\sum_{I'\subseteq I}\Norm{f_{I'}}{1}
	=\sum_{j\geq i}2^{j-i}2^{-j^2/p'-j-1/4}
	=2^{-i-1/4}\sum_{j\geq i}2^{-j^2/p'}\leq c_p\Norm{f_I}{1},
	\end{equation*}
	since there are $2^{j-i}$ intervals $I'\subseteq I$ of length $2^{-j}$, and since
	\begin{equation*}
	\sum_{j\geq i}2^{-j^2/p'}
	=\sum_{k\geq 0}2^{-(i+k)^2/p'}
	\leq 2^{-i^2/p'}\sum_{k\geq 0}2^{-k^2/p'}
	=c_p 2^{-i^2/p'}.
	\end{equation*}
	On the other hand, since the functions $f_{I'}$ are disjointly supported,
	\begin{equation*}
	\Norm{f_I^*}{p}^p
	=\sum_{I'\subseteq I}\Norm{f_{I'}}{p}^p
	=\sum_{j\geq i} 2^{j-i}h_j^p\ell_j
	=\sum_{j\geq i} 2^{j-i}2^{j^2}2^{-(j+1/2)^2}
	=2^{-i-1/4}\sum_{j\geq i}1=\infty.
	\end{equation*}
	\end{proof}

Next we start estimating the $JN_{p}$-norm of $f$. From now on, let $\mathscr J$ be a collection of pairwise disjoint intervals. Without loss of generality we may assume that $\mathscr J$ does not contain any intervals where $f$ is constant,  as such intervals do not contribute to the $JN_{p}$-norm.

\begin{lemma}
For an interval $J$, let
\begin{equation*}
  F(J):=\abs{J}^{1-p}\Big(\int_J\abs{f-\ave{f}_J}\Big)^{p}.
\end{equation*}
If $F(J)\neq 0$, there is a unique largest interval $I=I_J\in\mathscr D$ such that $J\cap \hat I\neq\varnothing$, and in this case $J$ must also intersect the boundary of $\hat I$.

In particular, for every $I\in\mathscr D$, there are at most two intervals $J\in\mathscr J$ such that $I=I_J$.
\end{lemma}

\begin{proof}
Since $f$ is supported on the union of the intervals $\hat I$, if $F(J)\neq 0$, clearly $J$ must intersect some $\hat I$. Suppose that $J$ intersects two intervals $\hat I_u$, $u=1,2$ of equal length. But then, by construction (see Figure \ref{picture}), there is a bigger interval $\hat I$ (where $I$ can be taken, as above,
to be the smallest dyadic interval containing $I_1$ and $I_2$) lying between the intervals $\hat I_u$, and $J$ must also intersect $\hat I$. This shows that there is a unique interval $\hat I$ of maximal length.

If $J\subseteq \hat I$, then $F(J)=0$, since $f$ is constant on $\hat I$. Thus $J$ must also intersect the complement of $\hat I$, and hence the boundary of $\hat I$.
The boundary of $\hat I$ has only two points, so a disjoint collection $\mathscr J$ can contain at most two intervals $J$ like this.
\end{proof}

\begin{lemma}\label{lem:FJ}
If $F(J)\neq 0$ and $I=I_J$, then
\begin{equation*}
  F(J)\leq c_p\Big[\abs{J}^{1-p}\Big(\int_{J\setminus\hat I}f\Big)^p+\min(\abs{J\cap\hat I},\abs{J\setminus\hat I})h_I^p\Big].
\end{equation*}
\end{lemma}

\begin{proof}
Recall that $f=h_I$ on $\hat I$. Thus
\begin{equation*}
  \ave{f}_J
  =\frac{1}{\abs{J}}\int_{J\setminus\hat I}f+\frac{\abs{J\cap\hat I}}{\abs{J}}h_I
\end{equation*}
and
\begin{equation*}
\begin{split}
  \int_J\abs{f-\ave{f}_J}
  &=\int_{J\setminus\hat I}\abs{f-\ave{f}_J}+\int_{J\cap\hat I}\Babs{\frac{\abs{J\setminus\hat I}}{\abs{J}}h_I-\frac{1}{\abs{J}}\int_{J\setminus\hat I}f} \\
  &\leq \int_{J\setminus\hat I}f+\abs{J\setminus\hat I}\ave{f}_J+\frac{\abs{J\cap\hat I}\abs{J\setminus\hat I}}{\abs{J}}h_I
    +\frac{\abs{J\cap\hat I}}{\abs{J}}\int_{J\setminus\hat I}f \\
   &\leq 2\int_{J\setminus\hat I}f+2\frac{\abs{J\cap\hat I}\abs{J\setminus\hat I}}{\abs{J}}h_I.
\end{split}
\end{equation*}
Hence
\begin{equation*}
\begin{split}
  F(J) &\leq \abs{J}^{1-p}\Big(2\int_{J\setminus\hat I}f+2\frac{\abs{J\cap\hat I}\abs{J\setminus\hat I}}{\abs{J}}h_I\Big)^p \\
  &\leq c_p\Big[\abs{J}^{1-p}\Big(\int_{J\setminus\hat I}f\Big)^p+\min(\abs{J\cap\hat I},\abs{J\setminus\hat I})h_I^p\Big].\qedhere
\end{split}
\end{equation*}
\end{proof}

Let $J$ be an interval with $F(J)\neq 0$, and $I=I_J$. We say that $J$ is
\begin{itemize}
  \item short, if $\abs{J\setminus\hat I}<\delta_I$;
  \item medium, if $\delta_I\leq \abs{J\setminus\hat I}<2D_I$;
  \item long, if $\abs{J\setminus\hat I}\geq 2D_I$.
\end{itemize}

\begin{lemma}
If $J$ is short and $I=I_J$, then
\begin{equation*}
  F(J)\leq c_p d_I h_I^p = c_p 2^{-2i}
\end{equation*}
and
\begin{equation*}
  \sum_{\substack{J\in\mathscr J \\ J\text{ short}}}F(J)\leq c_p.
\end{equation*}
\end{lemma}

\begin{proof}
In this case $J\setminus\hat I$ is too small to reach the support of $f$ outside $\hat I$, and hence $f=0$ on $J\setminus \hat I$. Thus only the second term from Lemma \ref{lem:FJ} contributes to $F(J)$. Thus
\begin{equation*}
  F(J)\leq c_p\abs{J\setminus\hat I}h_I^p
  \leq c_p\delta_I h_I^p\leq c_p d_I h_I^p
  =c_p 2^{-(i+1)^2}2^{i^2}=c_p 2^{-2i-1}.
\end{equation*}
Now
\begin{equation*}
  \sum_{\substack{J\in\mathscr J \\ J\text{ short}}}F(J)
  =\sum_{I\in\mathscr D}\sum_{\substack{J\in\mathscr J\text{ short} \\ I_J=I}}F(J)
  \leq \sum_{i=0}^\infty 2^i\cdot 2\cdot c_p 2^{-2i}=c_p.
\end{equation*}
since there are:
\begin{itemize}
  \item exactly $2^i$ intervals of $I\in\mathscr D$ of length $2^{-i}$;
  \item for each $I$, at most two short intervals $J\in\mathscr J$ with $I_J=I$; and
  \item for each $J$, we have the estimate for $F(J)$ as written.
\end{itemize}
\end{proof}

\begin{lemma}
If $J$ is medium and $I=I_J$, then
\begin{equation*}
  F(J)\leq c_p(2^{-ip}+2^{-2i})
\end{equation*}
and
\begin{equation*}
  \sum_{\substack{J\in\mathscr J \\ J\text{ medium}}}F(J)\leq c_p.
\end{equation*}
\end{lemma}

\begin{proof}
In this case, we need to consider both terms from Lemma \ref{lem:FJ}. However, the estimate for the second term is exactly as in the case of short intervals, since $2D_I\leq 6d_I$, and we concentrate on the first term. Letting $I_1,I_2$ be the two halves of $I$, we have
\begin{equation*}
  \int_{J\setminus\hat I}f
  \leq \sum_{u=1,2}\Norm{f_{I_u}^*}{1}
  \leq \sum_{u=1,2}c_p 2^{-(i+1)^2/p'-(i+1)}
\end{equation*}
Since $\abs{J}\geq\abs{J\setminus\hat I}\geq \delta_I\geq\frac12 d_I$ and $p/p'=p-1$, we get
\begin{equation*}
\begin{split}
  \abs{J}^{1-p}\Big(\int_{J\setminus\hat I}f\Big)^p
  &\leq c_p d_I^{1-p}   (2^{-(i+1)^2/p'-(i+1)})^p \\  
    &=c_p 2^{-(i+1)^2(1-p)}2^{-(i+1)^2(p-1)-(i+1)p}
  =c_p 2^{-ip}.
\end{split}
\end{equation*}
Taking into account the second term from Lemma \ref{lem:FJ}, which is the same as in the short case, we get
\begin{equation*}
  F(J)\leq c_p(2^{-ip}+2^{-2i})
\end{equation*}
As in the short case, this gives
\begin{equation*}
  \sum_{\substack{J\in\mathscr J \\ J\text{ medium}}}F(J)
  =\sum_{I\in\mathscr D}\sum_{\substack{J\in\mathscr J\text{ medium} \\ I_J=I}}F(J)
  \leq \sum_{i=0}^\infty 2^i\cdot 2\cdot c_p (2^{-pi}+2^{-2i})=c_p,
\end{equation*}
since $p>1$.
\end{proof}

\begin{lemma}
If $J$ is long and $I=I_J$, then
\begin{equation*}
  F(J)\leq c_p 2^{-i}=c_p\abs{I}.
\end{equation*}
The corresponding intervals $I$ form a Carleson family, in the sense that
\begin{equation*}
  \sum_{\substack{I\in\mathscr D \\ \exists J\text{ long: }I=I_J}}\abs{I}\leq 2,
\end{equation*}
and hence
\begin{equation*}
  \sum_{\substack{J\in\mathscr J \\ J\text{ long}}}F(J)\leq c_p.
\end{equation*}
\end{lemma}

\begin{proof}
If $J$ is long, then $\abs{J\setminus \hat I}> 2D_I$, and hence the part of $J$ on at least one side, say left (right), of $\hat I$ is longer than $D_I$. But this means that $J$ contains all the intervals $\hat I'$, where $I'\subsetneq I$ is contained in the left (right) half of $I$. Consequently, all dyadic subintervals $I'\subsetneq I$ of the form $I_{J'}$ for some $J'\in\mathscr J$, must be contained in just one half of $I$. Let us say that $I\in\mathscr D$ is long, if it is of the form $I_J$ for some long $J$. Hence, if $I\in\mathscr D$ is long, all its long subintervals are contained in just one half of $I$. If $I^*$ denotes the other half of $I$, then clearly $\abs{I^*}=\frac12\abs{I}$, and the intervals $I^*$ are pairwise disjoint for long intervals $I$. This immediately implies the Carleson property by
\begin{equation*}
  \sum_{I\in\mathscr D\text{ long}}\abs{I}
  =2\sum_{I\in\mathscr D\text{ long}}\abs{I^*}
  =2\Babs{\bigcup_{I\in\mathscr D\text{ long}}I^*}
  \leq 2\abs{[0,1)}=2.
\end{equation*}

As for the estimate for $F(J)$, we note that the part corresponding to the first term of Lemma \ref{lem:FJ} is just the same as in the medium case: by the maximality of $\hat I$, we know that $J\setminus \hat I$ cannot meet any other parts of $f$ than $f_{I'}^*$, where $I'$ is the left or right half of $I$, and in addition to this observation, we only used that $\abs{J\setminus I}$ is big enough. And of course $c_p 2^{-ip}\leq c_p 2^{-i}$.

We turn to the second term in Lemma \ref{lem:FJ}, which is estimated slightly differently from the previous cases. Namely, in the lack of a good bound for $\abs{J\setminus\hat I}$, we instead observe that $\abs{J\cap\hat I}\leq\abs{\hat I}=\ell_I$, and hence
\begin{equation*}
  c_p\min(\abs{J\cap\hat I},\abs{J\setminus\hat I})h_I^p
  \leq c_p\ell_I h_I^p
  =c_p 2^{-(i+1/2)^2}2^{i^2}=c_p 2^{-i}=c_p\abs{I},
\end{equation*}
as claimed. This would not be good enough to sum over all dyadic intervals, but instead the Carleson property comes to rescue:
\begin{equation*}
\begin{split}
  \sum_{J\in\mathscr J\text{ long}}F(J)
  \leq  \sum_{I\in\mathscr D\text{ long}} \sum_{\substack{ J\in\mathscr J\text{ long} \\ I_J=I}}F(J)
  \leq \sum_{I\in\mathscr D\text{ long}} c_p\abs{I}\leq c_p.
\end{split}
\end{equation*}
\end{proof}

\begin{proof}[Proof of Proposition \ref{prop:badf}]
Let $f$ be the function discussed above. We already checked that $\Norm{f}{p}=\infty$.

On the other hand, if $\mathscr J$ is a disjoint family of intervals, we checked that
\begin{equation*}
  \sum_{J\in\mathscr J}F(J)
  =\sum_{\substack{ J\in\mathscr J \\ J\text{ short}}}F(J)
    +\sum_{\substack{ J\in\mathscr J \\ J\text{ medium}}}F(J)
    +\sum_{\substack{ J\in\mathscr J \\ J\text{ long}}}F(J)
  \leq c_p+c_p+c_p.
\end{equation*}
By definition, this shows that
\begin{equation*}
  \Norm{f}{JN_p}\leq c_p.\qedhere
\end{equation*}
\end{proof}

\begin{remark}\label{rem:Lpq}
In fact, the same function $f$ also satisfies $\Norm{f}{L^{p,q}}=\infty$ for every $q<\infty$, proving that $JN_p\not\subset L^{p,q}$. This is seen as follows:

Recall that $f$ takes the value $h_i=2^{i^2/p}$ on $2^i$ disjoint intervals of length $\ell_i=2^{-(i+1/2)^2}$. Thus
\begin{equation*}
  \abs{\{f=h_i=2^{i^2/p}\}}=2^i\ell_i=2^i 2^{-i^2-i-1/4}=2^{-i^2-1/4}.
\end{equation*}
The norm in the Lorentz space $L^{p,q}$ is given by
\begin{equation*}
\begin{split}
  \Norm{f}{L^{p,q}}^q &= \int_0^\infty \big(t\abs{\{f> t\}}^{1/p}\big)^q\frac{\ud t}{t} \\
  &\geq \sum_{i=0}^\infty \int_{h_i}^{h_{i+1}} t^{q-1}\ud t\cdot  \abs{\{f= h_{i+1}\}}^{q/p} \\
  &\gtrsim \sum_{i=0}^\infty (h_{i+1}^q-h_i^q)\cdot 2^{-(i+1)^2 q/p} \\
  &=\sum_{i=0}^\infty (2^{(i+1)^2 q/p}-2^{i^2 q/p})\cdot 2^{-(i+1)^2 q/p} \\
  &=\sum_{i=0}^\infty (1-2^{-(2i+1)q/p}).
\end{split}
\end{equation*}
Since the $i$th term of the series converges to $1$ (and not $0$) as $i\to\infty$, the series is clearly not summable.
\end{remark}

\section{A multidimensional counterexample} 

In this short section we show how to lift the one-dimensional counterexample to several variables. It turns out that this can be achieved in a soft way by simply using the previous result as a black box, without revisiting any of the technical details. This is thanks to the following simple extension result that might have some independent interest:

\begin{proposition}\label{prop:constExtJNp}
Let $Q_0\subset\R^d$ be a cube, $f\in L^1(Q_0)$, and $\tilde f(x,t):=f(x)$ its trivial extension to $(x,t)\in \tilde Q_0:=Q_0\times[0,\ell(Q_0))\subset\R^{d+1}$. Then $f\in JN_p(Q_0)$ if and only if $\tilde f\in JN_p(\tilde Q_0)$, and
\begin{equation*}
  2^{-1/p}\Norm{f}{JN_p(Q_0)}\ell(Q_0)^{1/p}\leq\Norm{\tilde f}{JN_p(\tilde Q_0)}\leq\Norm{f}{JN_p(Q_0)}\ell(Q_0)^{1/p}.
\end{equation*}
\end{proposition}

\begin{proof}
Let first $f\in JN_p(Q_0)$, and let $\tilde Q_i=Q_i\times I_i$ be any disjoint subcubes of $\tilde Q_0$. Since $\tilde f$ is constant in the $t$-direction, we get $\ave{\tilde f}_{\tilde Q_i}=\ave{f}_{Q_i}$ and further
\begin{equation*}
\begin{split}
  \sum_i\abs{\tilde Q_i}\Big(\fint_{\tilde Q_i}\abs{\tilde f-\ave{\tilde f}_{\tilde Q_i}}\Big)^p
  &=\sum_i\abs{Q_i}\ell(Q_i)\Big(\fint_{Q_i}\abs{f-\ave{f}_{Q_i}}\Big)^p \\
  &=\sum_i\abs{Q_i}\int_0^{\ell(Q_0)}1_{I_i}(t)\ud t\Big(\fint_{Q_i}\abs{f-\ave{f}_{Q_i}}\Big)^p \\  
  &=\int_0^{\ell(Q_0)}\sum_{i:t\in I_i}\abs{Q_i}\Big(\fint_{Q_i}\abs{f-\ave{f}_{Q_i}}\Big)^p \ud t.
\end{split}
\end{equation*}
For every fixed $t$, the collection of cubes $Q_i$, such that $t\in I_i$, is disjoint. (In fact, if $t\in I_i\cap I_j$ and $x\in Q_i\cap Q_j$, then $(x,t)\in (Q_i\times I_i)\cap(Q_j\times I_j)=\tilde Q_i\cap \tilde Q_j$, a contradiction with the disjointness of the cubes $\tilde Q_i$.)
Thus
\begin{equation*}
  \sum_{i:t\in I_i}\abs{Q_i}\Big(\fint_{Q_i}\abs{f-\ave{f}_{Q_i}}\Big)^p\leq\Norm{f}{JN_p(Q_0)}^p
\end{equation*}
and hence
\begin{equation*}
  \sum_i\abs{\tilde Q_i}\Big(\fint_{\tilde Q_i}\abs{\tilde f-\ave{\tilde f}_{\tilde Q_i}}\Big)^p
  \leq\int_0^{\ell(Q_0)}\Norm{f}{JN_p(Q_0)}^p\ud t=\Norm{f}{JN_p(Q_0)}^p\ell(Q_0).
\end{equation*}
Since this is true for every collection of disjoint subcubes $\tilde Q_i\subset \tilde Q_0$, we get the second claimed bound.

Let then $\tilde f\in JN_p(\tilde Q_0)$, and let $Q_i\subset Q_0$ be disjoint subcubes. We consider the cubes $\tilde Q_{i,n}:=Q_i\times[(n-1)\ell(Q_i),n\ell(Q_i))$, where $1\leq n\leq N_i$, and $N_i$ is chosen so that $N_i\ell(Q_i)\leq\ell(Q_0)<(N_i+1)\ell(Q_i)$. These are disjoint subcubes of $\tilde Q_0$, and
\begin{equation*}
  N_i\ell(Q_i)\geq\max\{\ell(Q_i),\ell(Q_0)-\ell(Q_i)\}\geq\frac12\ell(Q_0).
\end{equation*}
Hence
\begin{equation*}
\begin{split}
  \Norm{\tilde f}{JN_p(\tilde Q_0)}^p
  &\geq\sum_i\sum_{n=1}^{N_i}\abs{\tilde Q_{i,n}}\Big(\fint_{\tilde Q_{i,n}}\abs{\tilde f-\ave{\tilde f}_{\tilde Q_{i,n}}}\Big)^p \\
  &=\sum_i\sum_{n=1}^{N_i}\abs{Q_i}\ell(Q_i)\Big(\fint_{Q_{i}}\abs{f-\ave{f}_{Q_{i}}}\Big)^p \\
  &\geq\frac{\ell(Q_0)}{2}\sum_i\abs{Q_i}\Big(\fint_{Q_{i}}\abs{f-\ave{f}_{Q_{i}}}\Big)^p,
\end{split}
\end{equation*}
where we used $\sum_{n=1}^{N_i}\ell(Q_i)=N_i\ell(Q_i)\geq\frac12\ell(Q_0).$
Since this holds for all disjoint collections of cubes $Q_i\subset Q_0$, we deduce the first claimed bound.
\end{proof}

\begin{corollary}
For every integer $d\geq 1$, there is a function $$f\in JN_p([0,1)^d)\setminus L^p([0,1)^d).$$
\end{corollary}

\begin{proof}
We have previously constructed such a function when $d=1$, with support in an interval of finite length. By rescaling, we can assume this function is supported in 
the interval $[0,1)$.  Suppose that the claim is true for some $d$, and let $f$ be the corresponding function. Let us then consider its trivial extension $\tilde f(x,t)=f(x)$ for $(x,t)\in[0,1)^d\times[0,1)$. By Proposition \ref{prop:constExtJNp}, we have $\Norm{\tilde f}{JN_p([0,1)^{d+1})}\leq\Norm{f}{JN_p([0,1)^d)}<\infty$, whereas clearly $\Norm{\tilde f}{L^p([0,1)^{d+1})}=\Norm{f}{L^p([0,1)^d)}=\infty$.
\end{proof}

\section{Equivalent norms}\label{sec:complements}

It might occur to one to consider the following generalisation of the $JN_p$ norm:
\begin{equation*}
  \Norm{f}{JN_{p,q}}^p
  :=\sup \sum_i \abs{Q_i}\Big(\fint_{Q_i}\abs{f-\ave{f}_{Q_i}}^q \Big)^{p/q},
\end{equation*}
where the supremum is as in the definition of $JN_p$. However, this would not yield anything new, as shown by the following:

\begin{proposition}
Let $Q\subset\R^d$ be a cube. Then
\begin{equation*}
  JN_{p,q}(Q)=\begin{cases} JN_p(Q), & 1\leq q<p, \\ L^q(Q), & p\leq q<\infty. \end{cases}
\end{equation*}
\end{proposition}

\begin{proof}
By H\"older's inequality it is clear that $JN_{p,q}\subset JN_p=JN_{p,1}$ for $q\geq 1$. Let $q\in(1,p)$ and consider disjoint cubes $Q_i$. Then
\begin{equation*}
  \Big(\fint_{Q_i}\abs{f-\ave{f}_{Q_i}}^q\Big)^{1/q}
  \lesssim\abs{Q_i}^{-1/p}\Norm{f-\ave{f}_{Q_i}}{L^{p,\infty}(Q_i)}
  \lesssim\abs{Q_i}^{-1/p}\Norm{f}{JN_p(Q_i)}
\end{equation*}
by the embedding $L^{p,\infty}\subset L^q$ and the John--Nirenberg lemma for $JN_p$. If we now choose disjoint subcubes $Q_{ij}\subset Q_i$ for which the $JN_p(Q_i)$ norm of $f$ is almost achieved, we have
\begin{equation*}
\begin{split}
    \sum_i &\abs{Q_i}\Big(\fint_{Q_i}\abs{f-\ave{f}_{Q_i}}^q\Big)^{p/q} 
  \lesssim \sum_i \Norm{f}{JN_p(Q_i)}^p \\
  &  \lesssim\sum_i\sum_j \abs{Q_{ij}}\Big(\fint_{Q_{ij}}\abs{f-\ave{f}_{Q_{ij}}}\Big)^p 
  \leq\Norm{f}{JN_p(Q)}^p,
\end{split}  
\end{equation*}
since $Q_{ij}$ are disjoint subcubes of $Q$. This shows that $JN_p\subset JN_{p,q}$ for $q<p$.

For $q\geq p$, by considering the trivial partition consisting of the single cube $Q$, we find that
\begin{equation*}
  \abs{Q}\Big(\fint_Q\abs{f-\ave{f}_Q}^q\Big)^{p/q}\leq\Norm{f}{JN_{p,q}(Q)}^p
\end{equation*}
so that $JN_{p,q}\subset L^q$. On the other hand, Jensen's inequality (applied to the convex combination with coefficients $\abs{Q_i}/\abs{Q}$) shows that
\begin{equation*}
\begin{split}
  \sum_i\frac{\abs{Q_i}}{\abs{Q}}\Big(\fint_{Q_i}\abs{f-\ave{f}_{Q_i}}^q\Big)^{p/q}
  &\lesssim \Big(\sum_i\frac{\abs{Q_i}}{\abs{Q}}\fint_{Q_i}\abs{f}^q\Big)^{p/q} \\
  &= \Big(\sum_i\frac{1}{\abs{Q}}\int_{Q_i}\abs{f}^q\Big)^{p/q}
    \leq \Big(\frac{1}{\abs{Q}}\int_{Q}\abs{f}^q\Big)^{p/q}
\end{split}
\end{equation*}
and hence $L^q\subset JN_{p,q}$ as well.
\end{proof}

\section{Duality}\label{sec:duality}
We now work on a fixed cube $Q_0$ in $\R^d$.
In analogy with the well-known $H^1$-$\BMO$ duality, one might expect to identify $JN_p$ with the dual of some ``Hardy kind of'' space, say $HK_{p'}$. Indeed, from standard duality arguments one can see that
\begin{equation*}
  \Norm{f}{JN_p(Q_0)}\eqsim
  \sup_g |\langle f, g \rangle|,
\end{equation*}
where  $\eqsim$ indicates that $\lesssim$ holds in both directions,
and the supremum is taken over all $g=\sum_j a_j$ associated to sequences of functions $a_j$
defined on disjoint subcubes $Q_j\subset Q_0$ and satisfying
\begin{equation*}
 a_j\in L^\infty_0(Q_j),\quad  \sum_j\abs{Q_j}\Norm{a_j}{\infty}^{p'}\leq 1.
\end{equation*}
 Here $L_0^s(Q_j)$, $1 \leq s \leq \infty$, denotes the space of $L^s$ functions on $Q_j$ with mean zero.  For such $g$ we can define
$$\langle f, g \rangle:=\sum_j \int_{Q_j} f a_j$$
whenever $f \in JN_p(Q_0)$.  This suggests that a predual of $JN_p$ might be a linear space generated by all functions $g$ of this form. Note that each $a_j$ is (maybe up to scaling) an \emph{atom} of the Hardy space $H^1$. We shall refer to functions $g$ as above by the name \emph{polymers}. (The word `molecule' already has a different established usage in the theory of Hardy spaces.)

In analogy with the notion of Hardy space $L^q$-atoms, $1 < q \leq \infty$, we make a slightly more general definition:

\begin{definition}\label{def:polymers}
Let $1<r<s\leq\infty$.
We say that $g$ is an $(r,s)$-polymer if $g=\sum_{j=1}^\infty a_j$ pointwise, where $a_j\in L^s_0(Q_j)$ for disjoint cubes $Q_j$ (note that the pointwise convergence of such a series is trivial by disjointness), and
\begin{equation*}
  \Norm{(a_j)_{j=1}^\infty}{(r,s)}=\Big(\sum_j\abs{Q_j}\Big[\fint_{Q_j}\abs{a_j}^s\Big]^{r/s}\Big)^{1/r}<\infty,
\end{equation*}
with the usual reinterpretation for $s=\infty$.  
We define $\Norm{g}{(r,s)}$ as the infimum of $\Norm{(a_j)_{j=1}^\infty}{(r,s)}$ over all such representations of $g$ as $\sum_{j=1}^\infty a_j$.
The functions $a_j$ making up $g$ will be called $s$-atoms.

We say that $g\in HK_{rs}(Q_0)$ if there is a representation, convergent in norm in $L^r(Q_0)$,
\begin{equation*}
  g=\sum_{i=0}^\infty g_i,
\end{equation*}
where each $g_i$ is an $(r,s)$-polymer (thus in $L^r(Q_0)$ as we check in Remark \ref{rem:Lr} below), and $\sum_{i=0}^\infty \Norm{g_i}{(r,s)}<\infty$. We define $\Norm{g}{HK_{rs}}$ as the infimum of such sums over all such representations.
\end{definition}

\begin{remark}
\label{rem:Lr}
It is immediate from Jensen's inequality that any $(r,s)$-polymer $g_i=\sum_j a_{ij}$ satisfies
\begin{equation*}
  \Norm{g_i}{r}
  =\Big(\sum_j\abs{Q_{ij}}\fint_{Q_{ij}}\abs{a_{ij}}^r\Big)^{1/r}
  \leq\Big(\sum_j\abs{Q_{ij}}\Big[\fint_{Q_{ij}}\abs{a_{ij}}^s\Big]^{r/s}\Big)^{1/r}
\end{equation*}
Taking the infimum over all representations, it follows that $\Norm{g_i}{r}\leq\Norm{g_i}{(r,s)}$.
Thus, if $g\in HK_{rs}(Q_{0})$, its polymeric representation $g=\sum_j g_j$ converges in $L^r(Q_0)$, and $g\in L^r(Q_0)$.

Moreover, $L^s_0(Q_0)$ is a dense subspace  of $HK_{rs}(Q_0)$:  note that every $g\in L^s_0(Q_0)$ is an $(r,s)$-polymer with a trivial expansion consisting of one $s$-atom, and hence $L^s_0(Q_0)\subset HK_{rs}(Q_0)$ with a continuous embedding.  The convergence of the $(r,s)$-polymeric expansion $g=\sum_{i=1}^\infty g_i$ and the atomic expansions $g_i=\sum_{j=1}^\infty a_{ij}$ in the norm of $HK_{rs}(Q_0)$ shows that the finite sums $\sum_{i=1}^M\sum_{j=1}^N a_{ij}$ are dense in $HK_{rs}(Q_0)$, and these finite sums belong to $L^s_0(Q_0)$, since each $s$-atom $a_{ij}$ belongs to this space. 
\end{remark}

\begin{remark}
One may find a certain analogy between Definition \ref{def:polymers} and the atomic description of $H^1$ on the bi-disc by Chang and Fefferman \cite{CF:80}. Namely, a generic function in the bi-disc-$H^1$ is expressed as a sum of certain bi-disc atoms, each of which is further decomposed into pieces called `elementary particles', just like our $HK_{rs}$ functions are sums of polymers, each of which is a sum of (usual $H^1$-)atoms. Aside from the two levels of the expansion, however, Definition \ref{def:polymers} has not much in common with the Chang--Fefferman atoms, so that copying their nomenclature (atoms and elementary particles instead of polymers and atoms) would seem more misleading than useful in our context. The functions $a_j$ in Definition \ref{def:polymers}, are precisely classical atoms, so it seems natural to adopt the name polymer for the larger structures built from them. In contrast, the Chang--Fefferman atoms, while being the larger structures in their expansion, have properties closely analogous to those of classical atoms, whereas their elementary particles have additional smoothness properties, which are neither present in the classical $H^1$ theory nor in our new spaces. Altogether, our $HK_{rs}$ spaces should be seen as a closer relative of the classical $H^1$ than its bi-disc version.
\end{remark}

Our duality results for $JN_p$ will ultimately rely on the well-known duality of the $L^p$ spaces. In order to have an access to the simple duality of the reflexive $L^p$ spaces (in contrast to the more complicated situation of $L^\infty$), we first establish the following reduction to finite indices in our candidate predual space:

\begin{proposition}\label{prop:HKspaces}
We have the coincidence of spaces $HK_{r\infty}(Q_{0})=HK_{rs}(Q_{0})$ with equivalence of norms for all $s\in(r,\infty)$.
\end{proposition}

Hence we can define $HK_r(Q_{0}):=HK_{r\infty}(Q_{0})$ with just one index. We begin with a convenient decomposition lemma for a single atom; this is essentially known from the theory of the Hardy space $H^1$ (cf. \cite{GCRF}, Theorems III.3.6 and III.3.7), but we provide the details in a form convenient for our needs.

\begin{lemma}
Fix a constant $C>2^d$. Let $f\in L^1_0(Q_0)$ and $\lambda\geq\fint_{Q_0}\abs{f}$. Then
\begin{equation*}
  f=\sum_{k=0}^\infty\sum_j a_{kj},
\end{equation*}
where $a_{kj}\in L^\infty_0(Q_{k,j})$ for cubes $Q_{k,j}$ such that $Q_{k,j}\cap Q_{k,j'}=\varnothing$ for $j\neq j'$ and
\begin{equation*}
   \bigcup_j Q_{k,j}=\Big\{ M_{Q_0}f>C^k\lambda \Big\}\quad\forall k\geq 1,\quad \bigcup_j Q_{0,j}=Q_0,
\end{equation*}
where $M_{Q_0}$ is the maximal function related to the dyadic subcubes of the cube $Q_0$, and
\begin{equation*}
  \Norm{a_{kj}}{\infty}\lesssim C^k\lambda.
\end{equation*}
\end{lemma}

\begin{proof}
For each $k\geq 1$, let $Q_{k,j}$ be the maximal dyadic subcubes of $Q_0$ such that $\ave{\abs{f}}_{Q_{k,j}}>C^k\lambda$. By maximality and doubling, we also have
\begin{equation*}
   \ave{\abs{f}}_{Q_{kj}}\leq 2^{d}C^k \lambda<C^{k+1}\lambda
\end{equation*}
since $C>2^{d}$, so that any given cube can appear in at most one ``level'' $k$. We also define $Q_{0,j}:=Q_0$.

We can then write
\begin{equation*}
\begin{split}
  f &=1_{Q_0}(f-\ave{f}_{Q_0})=\sum_{k=0}^\infty\sum_j a_{kj},\\
   a_{kj} &:=1_{Q_{k,j}\setminus\bigcup_i Q_{k+1,i}}(f-\ave{f}_{Q_{k,j}})+\sum_{i:Q_{k+1,i}\subset Q_{k,j}}1_{Q_{k+1,i}}(\ave{f}_{Q_{k+1,i}}-\ave{f}_{Q_{k,j}}).
\end{split}
\end{equation*}
In fact, this is a basic telescoping identity when there exists a maximal $k$ with $x\in Q_{k,j}$ (and hence $x\notin Q_{k+1,i}$ for any $i$). On the other hand, if there are arbitrary large $k$ with $x\in Q_{k,j}$, then $M_{Q_0}f(x)=\infty$, and the set of such points has measure zero, hence is of no concern to us.

It is straightforward that $\Norm{a_{kj}}{\infty}\lesssim C^k\lambda$ and $\int a_{kj}=0$. Moreover, $a_{kj}$ is supported on $Q_{k,j}$, and these cubes are disjoint for fixed $k$ with $\bigcup_j Q_{k,j}=\{M_{Q_0}f>C^k\lambda\}$.
\end{proof}

\begin{proof}[Proof of Proposition \ref{prop:HKspaces}]
Since every $(r,\infty)$-polymer is \emph{a fortiori} an $(r,s)$-polymer, it is enough to prove that every $(r,s)$-polymer $g$ can
be decomposed into $(r,\infty)$-polymers, namely $g=\sum_{k=1}^\infty g_k$, with $\sum_{k=1}^\infty \Norm{g_k}{(r,\infty)}\lesssim\Norm{g}{(r,s)}$. 

By definition, we have $g=\sum_{\ell=1}^\infty A_\ell$, with atoms $A_\ell\in L^s_0(Q_\ell)$ supported on disjoint cubes and such that $\Norm{(A_\ell)_{\ell=1}^\infty}{(r,s)}=\sum_\ell\abs{Q_\ell}(\fint_{Q_\ell}\abs{A_\ell}^s)^{r/s}\lesssim\Norm{g}{(r,s)}^r$.

For each $\ell$, we apply the decomposition of the previous lemma to $A_\ell \in L_0^s(Q_\ell)$ in place of $f$, with $\lambda:=(\fint_{Q_\ell}\abs{A_\ell}^s)^{1/s}$. This leads to
\begin{equation*}
A_\ell=\sum_{k=0}^\infty\sum_j a_{kj}^\ell,
\end{equation*}
where $a_{kj}^\ell\in L^\infty_0(Q_{k,j}^\ell)$ for disjoint (in $j$) cubes $Q_{k,j}^\ell\subset Q_\ell$ such that
\begin{equation*}
  \bigcup_{j}Q_{k,j}^\ell=\Big\{ M_{Q_\ell} A_\ell>C^k(\fint_{Q_\ell}\abs{A_\ell}^s)^{1/s}\Big\}\quad\forall k\geq 1,\quad \bigcup_{j}Q_{0,j}^\ell=Q_\ell,
\end{equation*}
and  $\Norm{a_{kj}^\ell}{\infty}\lesssim C^k(\fint_{Q_\ell}\abs{A_\ell}^s)^{1/s}$. Since the cubes $Q_\ell$ are disjoint, it follows that the cubes $Q_{k,j}^\ell$ are disjoint when both $\ell$ and $j$ are allowed to vary, for each fixed $k$. Thus the following function, defined pointwise by
\begin{equation*}
g_k:=\sum_{\ell,j}a_{kj}^\ell,
\end{equation*}
satisfies the atomic structure requirements for an $(r,\infty)$-polymer, and it remains to check the relevant norm estimates. That is, we need to estimate
\begin{equation*}
\begin{split}
  \sum_{k=0}^\infty\Norm{g_k}{(r,\infty)}
  &\leq\sum_{k=0}^\infty\Big(\sum_{\ell,j}\abs{Q_{k,j}^\ell}\Norm{a_{kj}^\ell}{\infty}^r\Big)^{1/r} \\
  &\lesssim \sum_{k=0}^\infty C^k \Big(\sum_{\ell,j}\abs{Q_{k,j}^\ell}(\fint_{Q_\ell}\abs{A_\ell}^s)^{r/s}\Big)^{1/r} \\
  &\lesssim  \sum_{k=1}^\infty C^k \Big[\sum_\ell\Babs{\Big\{ M_{Q_\ell}A_\ell>C^k(\fint_{Q_\ell}\abs{A_\ell}^s)^{1/s}\Big\}}(\fint_{Q_\ell}\abs{A_\ell}^s)^{r/s}\Big]^{1/r} \\
  &\qquad+ \Big(\sum_\ell\abs{Q_\ell}(\fint_{Q_\ell}\abs{A_\ell}^s)^{r/s}\Big)^{1/r}, 
\end{split}
\end{equation*}
where we separated the easier term corresponding to $k=0$. Note that this last term is $\Norm{(A_\ell)_{\ell=1}^\infty}{(r,s)}\lesssim\Norm{g}{(r,s)}$, so it remains to estimate the sum over $k\geq 1$.

Exploiting the fact that $s>r$, we write $C^k=C^{-k\eps}C^{ks/r}$, where $\eps=s/r-1>0$. Then repeated use of H\"older's inequality gives
\begin{equation*}
\begin{split}
  &\sum_{k=1}^\infty C^{-k\eps}C^{ks/r} \Big[\sum_\ell\Babs{\Big\{ M_{Q_\ell} A_\ell>C^k(\fint_{Q_\ell}\abs{A_\ell}^s)^{1/s}\Big\}}(\fint_{Q_\ell}\abs{A_\ell}^s)^{r/s}\Big]^{1/r} \\
  &\leq\Big(\sum_{k=1}^\infty C^{-k\eps r'}\Big)^{1/r'}
    \Big[\sum_{k=1}^\infty C^{ks} \sum_\ell\Babs{\Big\{ M_{Q_\ell} A_\ell>C^k(\fint_{Q_\ell}\abs{A_\ell}^s)^{1/s}\Big\}}(\fint_{Q_\ell}\abs{A_\ell}^s)^{r/s}\Big]^{1/r} \\
   &\lesssim \Big[\sum_\ell  (\fint_{Q_\ell}\abs{A_\ell}^s)^{(r-s)/s}
      \sum_{k=1}^\infty \Big[C^{k} (\fint_{Q_\ell}\abs{A_\ell}^s)^{1/s}\Big]^s \Babs{\Big\{ M_{Q_\ell} A_\ell>C^k(\fint_{Q_\ell}\abs{A_\ell}^s)^{1/s}\Big\}}\Big]^{1/r} \\
    &\lesssim \Big[\sum_\ell  (\fint_{Q_\ell}\abs{A_\ell}^s)^{(r-s)/s} \Norm{M_{Q_\ell}A_\ell}{L^s}^s\Big]^{1/r}\qquad
       \text{by Cavalieri's principle} \\
    &\lesssim \Big[\sum_\ell  (\fint_{Q_\ell}\abs{A_\ell}^s)^{(r-s)/s} (\fint_{Q_\ell}\abs{A_\ell}^s)\abs{Q_\ell}\Big]^{1/r}\qquad
    \text{by the boundedness of $M_{Q_\ell}$ on $L^s$}\\
    &= \Big[\sum_\ell  (\fint_{Q_\ell}\abs{A_\ell}^s)^{r/s} \abs{Q_\ell}\Big]^{1/r}
      =\Norm{(A_\ell)_{\ell=1}^\infty}{(r,s)}\lesssim\Norm{g}{(r,s)}.
\end{split}
\end{equation*}
This completes the proof.
\end{proof}

Now we are ready for our main result about duality:

\begin{theorem}\label{thm:duality}
Let $r\in(1,\infty)$. Then $(HK_r(Q_0))^*\simeq JN_{r'}(Q_0)$ in the sense that:
\begin{enumerate}
  \item Every $f\in JN_{r'}(Q_0)$ induces a linear functional $\Lambda_f\in (HK_r(Q_0))^*$ of norm
\begin{equation*}
  \Norm{\Lambda_f}{(HK_r)^*}\eqsim\Norm{f}{JN_{r'}}
\end{equation*}
   by
\begin{equation}\label{eq:Lambdafg}
  \Lambda_f g=\sum_{i=1}^\infty\langle f,g_i \rangle
  =\sum_{i=1}^\infty \sum_j\int_{Q_{ij}} f a_{ij} ,
\end{equation}
whenever
\begin{equation*}
   g=\sum_{i=1}^\infty g_i,\qquad g_i=\sum_{j=1}^\infty a_{ij},
\end{equation*}
where $g_i$ are $(r,s)$-polymers for some $s\in(r,\infty]$ and $a_{ij}\in L^s_0(Q_{ij})$ for disjoint cubes $\{Q_{ij}\}_j$ are $s$-atoms such that
\begin{equation*}
  \sum_{i=1}^\infty\Norm{g_i}{(r,s)}\leq \sum_{i=1}^\infty\Norm{(a_{ij})_{j=1}^\infty}{(r,s)}<\infty.
\end{equation*}
 The value of $\Lambda_f g$ above is independent of the chosen representation of $g$ as a series of polymers, and their representations as series of atoms, as long as the finiteness condition above is satisfied. In fact, an alternative way of computing $\Lambda_f g$ is also given by
\begin{equation*}
    \Lambda_f g=\lim_{N\to\infty}\int_{Q_0}f_N g,\quad f_N(x):=\begin{cases} f(x), & \text{if }\abs{f(x)}\leq N, \\ \frac{N}{\abs{f(x)}}f(x), & \text{if }\abs{f(x)}>N, \end{cases}
\end{equation*}
where the integrals and the limit exists for all $f\in JN_{r'}(Q_0)$ and $g\in HK_r(Q_0)$.
 \item Every continuous linear functional $\Lambda$ on $HK_r(Q_0)$ has the form $\Lambda=\Lambda_f$ for some $f\in JN_{r'}(Q_0)$.
\end{enumerate}
\end{theorem}

The argument below will follow the broad outline of the proof of the $H^1$-$\BMO$ duality as given in \cite{Meyer:90}, Section 5.6. See also Theorem 1 in Ch. IV of \cite{St:93}. It might be noted that the question of well-definedness of the expression in \eqref{eq:Lambdafg} is somewhat serious in view of the examples given in \cite{Bownik:05} in the context of $H^1$.

\begin{proof}
(a) Let first $g$ be an $(r,s)$-polymer for $1<r<s\leq\infty$. Then $1\leq s'<r'<\infty$ and $f\in JN_{r'}(Q_0)\simeq JN_{r',s'}(Q_0)$. Let $g=\sum_{j=1}^\infty a_j$ with $a_j\in L^s_0(Q_j)$, where $\sum_j\abs{Q_j}(\fint_{Q_j}\abs{a_j}^s)^{r/s}\lesssim\Norm{g}{(r,s)}^r$ and the cubes $Q_j\subset Q_0$ are disjoint. Then $ \langle f,g \rangle
 :=\sum_j\int_{Q_j}fa_j$, where the sum converges absolutely and
\begin{equation*}
\begin{split}
 |\langle f,g \rangle|
  &\leq \sum_j\Babs{\int_{Q_j}fa_j}
  \leq \sum_j\abs{Q_j}\Babs{\fint_{Q_j}(f-\ave{f}_{Q_j})a_j} \\
  &\leq\sum_j\abs{Q_j}\Big(\fint_{Q_j}\abs{f-\ave{f}_{Q_j}}^{s'}\Big)^{1/s'}\Big(\fint_{Q_j}\abs{a_j}^s\Big)^{1/s} \\
  &\leq\Big[\sum_j\abs{Q_j}\Big(\fint_{Q_j}\abs{f-\ave{f}_{Q_j}}^{s'}\Big)^{r'/s'}\Big]^{1/r'}\Big[\sum_j\abs{Q_j}\Big(\fint_{Q_j}\abs{a_j}^s\Big)^{r/s}\Big]^{1/r} \\
  &\lesssim\Norm{f}{JN_{r',s'}}\Norm{g}{(r,s)} \lesssim \Norm{f}{JN_{r'}}\Norm{g}{(r,s)}.
\end{split}
\end{equation*}
If $g\in HK_r(Q_0)\simeq HK_{rs}(Q_0)$ is a sum of $(r,s)$-polymers $g=\sum_{i=1}^\infty g_i$, with $\sum_{i=1}^\infty\Norm{g_i}{(r,s)}\lesssim\Norm{g}{HK_{r}}$, the estimate above applied to each $g_i$ gives the absolute convergence of the series $\sum_{i=1}^\infty \langle f,g_i \rangle$ with the bound
\begin{equation}\label{eq:LambdaLessJN}
  \abs{\Lambda_f g}\leq\sum_{i=1}^\infty|\langle f,g_i \rangle|\lesssim \sum_{i=1}^\infty\Norm{f}{JN_{r'}}\Norm{g_i}{(r,s)}\lesssim\Norm{f}{JN_{r'}}\Norm{g}{HK_r}.
\end{equation}

(b) To show that $\Lambda_f g$ is independent of the expansion of $g$ and thus well-defined, we derive an alternative representation for $\Lambda_f g$. For a function $f$ and a number $N>0$, let
\begin{equation*}
  f_N(x):=\begin{cases} f(x), & \text{if }\abs{f(x)}\leq N, \\ \frac{N}{\abs{f(x)}}f(x), & \text{if }\abs{f(x)}>N. \end{cases}
\end{equation*}
Then a well-known estimate from the standard $\BMO$ theory says that
\begin{equation*}
  \Big(\fint_Q\abs{f_N-\ave{f_N}_Q}^q\Big)^{1/q}\lesssim
  \Big(\fint_Q\abs{f-\ave{f}_Q}^q\Big)^{1/q},
\end{equation*}
which implies that $\Norm{f_N}{JN_{r'}}\lesssim\Norm{f}{JN_{r'}}$ uniformly in $N$. On the other hand, it is clear that $f_N\in L^\infty(Q_0)$. From Remark~\ref{rem:Lr}, we see that the polymeric expansion $g=\sum_{i=1}^\infty g_i\in HK_{r}(Q_0)$ converges in $L^r(Q_0)$ and hence in $L^1(Q_0)$. Thus the product $f_N g$ is integrable, and
\begin{equation*}
  \int_{Q_0}f_Ng=\sum_i\int_{Q_0}f_N g_i
  =\sum_i\sum_j\int_{Q_{i,j}} f_N a_{ij},
\end{equation*}
where $g_i=\sum_j a_{ij}$ is a disjoint atomic expansion of the polymer $g_i$. Here $a_{ij}\in L^s_0(Q_{i,j})$, while $f\in JN_{r'}(Q_0)\subset L^{s'}(Q_0)$ since $s' < r'$. Hence $\int_{Q_{i,j}}f_N a_{ij}\to\int_{Q_{i,j}}f a_{ij}$ as $N\to\infty$ by dominated convergence. On the other hand, we have
\begin{equation*}
 \Babs{\int_{Q_{i,j}}f_N a_{ij}}\lesssim \abs{Q_{i,j}} \Big(\fint_{Q_{i,j}}\abs{f-\ave{f}_{Q_{i,j}} }^{s'}\Big)^{1/s'} \Big(\fint_{Q_{i,j}}\abs{a_{ij}}^s\Big)^{1/s},
\end{equation*}
where the right side is independent of $N$ and summable over $j$ with a bound that is a constant multiple of $\Norm{f}{JN_{r'}}\Norm{(a_{ij})_{j=1}^\infty}{(r,s)}$. Thus
\begin{equation*}
  \int_{Q_0} f_N g_i=\sum_j\int_{Q_{i,j}}f_N a_{ij}
  \to \sum_j\int_{Q_{i,j}}f a_{ij}=\langle f, g_i \rangle,
\end{equation*}
again by dominated convergence (of the sum). Since further $\Norm{f}{JN_{r'}}\Norm{(a_{ij})_{j=1}^\infty}{(r,s)}$ is summable over $i$ by assumption, yet another application of dominated convergence proves that
\begin{equation*}
  \Lambda_f g=\sum_i\langle f, g_i \rangle=\lim_{N\to\infty}\sum_i\int_{Q_0} f_N g_i
  =\lim_{N\to\infty}\int_{Q_0}f_N g.
\end{equation*}
But here the right side makes no reference to any expansion (either of $g$ in terms of polymers, or their expansion in terms of atoms), so that this quantity is manifestly independent of any such representation.

(c) Let finally $\Lambda\in (HK_{r}(Q_0))^*$ be given, and fix some $s\in(r,\infty)$. Since $HK_r(Q_0)\simeq HK_{rs}(Q_0)$, Remark~\ref{rem:Lr} shows $L^s_0(Q_0)$ is a dense subspace of $HK_r(Q_0)$. Let $\tilde\Lambda$ be the restriction of $\Lambda$ to $L^s_0(Q_0)$. Then $\tilde\Lambda\in(L^s_0(Q_0))^*$, and hence (by the well-known duality of the $L^p$ spaces) $\tilde\Lambda$ has a representation
\begin{equation*}
  \tilde\Lambda g=\int_{Q_0}fg
\end{equation*}
for some $f\in L^{s'}(Q_0)$.

We check that $f\in JN_{r'}(Q_0)$. Let $(Q_j)_{j=1}^N$ be a finite sequence of disjoint subcubes $Q_j\subset Q_0$. For every $j$, we have
\begin{equation*}
  \Big(\fint_{Q_j}\abs{f-\ave{f}_{Q_j}}^{s'}\Big)^{1/s'}
  = \sup\Big\{\fint_{Q_j}f b_j: b_j\in L^s_0(Q_j), \fint_{Q_j}\abs{b_j}^s= 1\Big\}.
\end{equation*}
For every $j$, we pick some $b_j$ that almost achieves the supremum. Set $g = \sum_j \lambda_j b_j$, where the positive numbers $\lambda_j$ are chosen so that $\sum_{j=1}^N\abs{Q_j}\lambda_j^r=1$ and
\begin{equation*}
\begin{split}
  &\Big(\sum_j\abs{Q_j}\Big(\fint_{Q_j}\abs{f-\ave{f}_{Q_j}}^{s'}\Big)^{r'/s'}\Big)^{1/r'} \\
  &=\sum_j\abs{Q_j}\lambda_j\Big(\fint_{Q_j}\abs{f-\ave{f}_{Q_j}}^{s'}\Big)^{1/s'} \\
  &\eqsim\sum_j\abs{Q_j}\lambda_j \fint_{Q_j}f b_j 
  =\int_{Q_0}f\sum_j \lambda_j b_j=\int_{Q_0}f g =\tilde\Lambda g =\Lambda g.
\end{split}
\end{equation*}
Note that each $b_j\in L^s_0(Q_0)$, and hence $g\in L^s_0(Q_0)$. Moreover, $g$ is an $(r,s)$-polymer with
\begin{equation*}
  \Norm{g}{(r,s)}\leq\Big[\sum_j\abs{Q_j}\Big(\fint_{Q_j}\abs{\lambda_j b_j}^s\Big)^{r/s}\Big]^{1/r}
  =\Big[\sum_j\abs{Q_j}\lambda_j^r\Big]^{1/r}=1.
\end{equation*}
Hence also $\Norm{g}{HK_r}\leq 1$, and therefore $\abs{\Lambda g}\leq\Norm{\Lambda}{(HK_r)^*}$. This gives
\begin{equation*}
  \Big(\sum_j\abs{Q_j}\Big(\fint_{Q_j}\abs{f-\ave{f}_{Q_j}}^{s'}\Big)^{r'/s'}\Big)^{1/r'}
  \lesssim\Norm{\Lambda}{(HK_r)^*}
\end{equation*}
for all finite families of disjoint cubes $Q_j$, from which the same estimate for countable families is immediate. Thus $f\in JN_{r'}(Q_0)$ and
\begin{equation}\label{eq:JNLessLambda}
  \Norm{f}{JN_{r'}}\lesssim\Norm{\Lambda}{(HK_r)^*}.
\end{equation}

(d) By parts (a) and (b) of the proof, we know that the $f$ above induces a functional $\Lambda_f$ on $HK_r(Q_0)$. If $g\in L^s_0(Q_0)$ then $g$ is an $(r,s)$-polymer consisting of one atom and so by definition $\Lambda_f g=\langle f, g \rangle = \int_{Q_0}fg$. But, for such functions, we also have $\Lambda g=\tilde\Lambda g=\int_{Q_0}fg$. Hence the functionals $\Lambda$ and $\Lambda_f$ agree on the subspace $L^s_0(Q_0)$ of $HK_r(Q_0)$. As noted above, this subspace is dense in $HK_r(Q_0)$. Since the continuous linear functionals $\Lambda$ and $\Lambda_f$ agree on a dense subspace, they must be equal. Combining the bounds \eqref{eq:LambdaLessJN} and \eqref{eq:JNLessLambda}, we also see that
\begin{equation*}
  \Norm{\Lambda_f}{(HK_r)^*}\lesssim\Norm{f}{JN_{r'}}\lesssim\Norm{\Lambda_f}{(HK_r)^*},
\end{equation*}
so that the two norms are equivalent.
\end{proof}

\begin{remark}
The only place in the above argument where the properties of real numbers beyond their Banach space structure were used was the representation of $\tilde\Lambda\in(L^s_0(Q_0))^*$ by a function $f\in L^{s'}(Q_0)$. For functions taking values in a Banach space $X$, the duality $(L^s(Q_0;X))^*=L^s(Q_0;X^*)$, for $1<s<\infty$, remains valid under the assumption that $X^*$ has the so-called \emph{Radon--Nikod\'ym property} (see \cite{HNVW}, Definitions 1.3.9, 1.3.27 and Theorems 1.3.10, 1.3.26). By the proof above, the duality $(HK_r(Q_0;X))^*\simeq JN_{r'}(Q_0;X^*)$, for $1<r<\infty$, remains valid under the same assumption.
\end{remark}


\begin{thebibliography}{10}

\bibitem{ABKY:11}
D.~Aalto, L.~Berkovits, O.~E. Kansanen, and H.~Yue.
\newblock John-{N}irenberg lemmas for a doubling measure.
\newblock {\em Studia Math.}, 204(1):21--37, 2011.

\bibitem{ABBF:14}
L.~Ambrosio, J.~Bourgain, H.~Brezis, and A.~Figalli.
\newblock Perimeter of sets and {$BMO$}-type norms.
\newblock {\em C. R. Math. Acad. Sci. Paris}, 352(9):697--698, 2014.

\bibitem{ABBF:16}
L.~Ambrosio, J.~Bourgain, H.~Brezis, and A.~Figalli.
\newblock B{MO}-type norms related to the perimeter of sets.
\newblock {\em Comm. Pure Appl. Math.}, 69(6):1062--1086, 2016.

\bibitem{BKM:16}
L.~Berkovits, J.~Kinnunen, and J.~M. Martell.
\newblock Oscillation estimates, self-improving results and good-{$\lambda$}
  inequalities.
\newblock {\em J. Funct. Anal.}, 270(9):3559--3590, 2016.

\bibitem{BBM:15}
J.~Bourgain, H.~Brezis, and P.~Mironescu.
\newblock A new function space and applications.
\newblock {\em J. Eur. Math. Soc. (JEMS)}, 17(9):2083--2101, 2015.

\bibitem{Bownik:05}
M.~Bownik.
\newblock Boundedness of operators on {H}ardy spaces via atomic decompositions.
\newblock {\em Proc. Amer. Math. Soc.}, 133(12):3535--3542 (electronic), 2005.

\bibitem{C:66}
S.~Campanato.
\newblock Su un teorema di interpolazione di {G}. {S}tampacchia.
\newblock {\em Ann. Scuola Norm. Sup. Pisa (3)}, 20:649--652, 1966.

\bibitem{CF:80}
S.-Y.~A. Chang and R.~Fefferman.
\newblock A continuous version of duality of {$H^{1}$}\ with {BMO} on the
  bidisc.
\newblock {\em Ann. of Math. (2)}, 112(1):179--201, 1980.

\bibitem{CR:80}
R.~R. Coifman and R.~Rochberg.
\newblock Another characterization of {BMO}.
\newblock {\em Proc. Amer. Math. Soc.}, 79(2):249--254, 1980.

\bibitem{FPW:98}
B.~Franchi, C.~P{\'e}rez, and R.~L. Wheeden.
\newblock Self-improving properties of {J}ohn-{N}irenberg and {P}oincar\'e
  inequalities on spaces of homogeneous type.
\newblock {\em J. Funct. Anal.}, 153(1):108--146, 1998.

\bibitem{GCRF}
J.~Garc{\'{\i}}a-Cuerva and J.~L. Rubio~de Francia.
\newblock {\em Weighted norm inequalities and related topics}, volume 116 of
  {\em North-Holland Mathematics Studies}.
\newblock North-Holland Publishing Co., Amsterdam, 1985.
\newblock Notas de Matem{\'a}tica [Mathematical Notes], 104.

\bibitem{Herz}
C.~Herz.
\newblock Bounded mean oscillation and regulated martingales.
\newblock {\em Trans. Amer. Math. Soc.}, 193:199--215, 1974.

\bibitem{HMV14}
R.~Hurri-Syrj{\"a}nen, N.~Marola, and A.~V. V{\"a}h{\"a}kangas.
\newblock Aspects of local-to-global results.
\newblock {\em Bull. Lond. Math. Soc.}, 46(5):1032--1042, 2014.

\bibitem{HNVW}
T.~{Hyt\"onen}, J.~{van Neerven}, M.~{Veraar}, and L.~{Weis}.
\newblock {\em {Analysis in Banach spaces. Volume I. Martingales and
  Littlewood-Paley theory.}}
\newblock Cham: Springer, 2016.

\bibitem{JN:61}
F.~John and L.~Nirenberg.
\newblock On functions of bounded mean oscillation.
\newblock {\em Comm. Pure Appl. Math.}, 14:415--426, 1961.

\bibitem{MP:98}
P.~MacManus and C.~P{\'e}rez.
\newblock Generalized {P}oincar\'e inequalities: sharp self-improving
  properties.
\newblock {\em Internat. Math. Res. Notices}, 1998(2):101--116, 1998.

\bibitem{MS:16}
N.~Marola and O.~Saari.
\newblock Local to global results for spaces of {$BMO$} type.
\newblock {\em Math. Z.}, 282(1-2):473--484, 2016.

\bibitem{Meyer:90}
Y.~Meyer.
\newblock {\em Wavelets and operators}, volume~37 of {\em Cambridge Studies in
  Advanced Mathematics}.
\newblock Cambridge University Press, Cambridge, 1992.
\newblock Translated from the 1990 French original by D. H. Salinger.

\bibitem{M:16}
M.~Milman.
\newblock Marcinkiewicz spaces, {G}arsia-{R}odemich spaces and the scale of
  {J}ohan-{N}irenberg self improving inequalities.
\newblock {\em Ann. Acad. Sci. Fenn. Math.}, 41(1):491--501, 2016.

\bibitem{S:65}
G.~Stampacchia.
\newblock The spaces {${\mathcal L}^{(p,\lambda )}$}, {$N^{(p,\lambda )}$} and
  interpolation.
\newblock {\em Ann. Scuola Norm. Sup. Pisa (3)}, 19:443--462, 1965.

\bibitem{St:93}
E.~M. Stein.
\newblock {\em Harmonic analysis: real-variable methods, orthogonality, and
  oscillatory integrals}, volume~43 of {\em Princeton Mathematical Series}.
\newblock Princeton University Press, Princeton, NJ, 1993.
\newblock With the assistance of Timothy S. Murphy, Monographs in Harmonic
  Analysis, III.

\end{thebibliography}

\end{document}